\documentclass[12pt,reqno]{amsart}
\usepackage{amsmath,amsfonts,hyperref}
\usepackage{amssymb}
\usepackage{amscd}
\usepackage{amsthm}
\usepackage{yhmath}
\usepackage{verbatim}

\usepackage{epsfig,subfigure}

\usepackage{todonotes}

\usepackage{lineno,xcolor}

\usepackage[all]{xy}
\usepackage{color} 
\usepackage{soul}

\numberwithin{equation}{section}

\newtheorem{theorem}{Theorem}[section]

\newtheorem{thm}[theorem]{Theorem}
\newtheorem{proposition}[theorem]{Proposition}

\newtheorem{lemma}[theorem]{Lemma}

\theoremstyle{remark}
\newtheorem{remark}[theorem]{Remark}

\theoremstyle{definition}
\newtheorem{definition}[theorem]{Definition}

\def\XXint#1#2#3{{\setbox0=\hbox{$#1{#2#3}{\int}$}
	\vcenter{\hbox{$#2#3$}}\kern-.5\wd0}}

\newcommand{\N}{\mathbb N}

\newcommand{\R}{\mathbb R}

\newcommand{\Z}{\mathbb Z}

\renewcommand{\H}{\mathcal H}

\renewcommand{\S}{{\mathcal S}}

\newcommand{\I}{{\mathcal I}}

\newcommand{\diam}{\operatorname{diam}}

\newcommand{\id}{\operatorname{id}}

\newcommand{\norm}[1]{\left\Vert#1\right\Vert}

\DeclareMathOperator{\Ext}{Ext}

\newcommand{\x}{{\bf x}}

\newcommand{\SL}[1]{{\color{red} [SL: #1]}}

\begin{document}

\title[Regular distances on the Heisenberg group]
{Ahlfors-regular distances on the Heisenberg group without biLipschitz pieces}

\author{Enrico Le Donne}
\author{Sean Li}
\author{Tapio Rajala}

\address[Le Donne, Rajala]{University of Jyvaskyla\\
         Department of Mathematics and Statistics \\
         P.O. Box 35 (MaD) \\
         FI-40014 University of Jyvaskyla \\
         Finland}
\email{enrico.e.ledonne@jyu.fi}
\email{tapio.m.rajala@jyu.fi}

\address[Li]{
Department of Mathematics\\
The University of Chicago\\
Chicago\\
IL 60637}
\email{seanli@math.uchicago.edu}

\thanks{S.L. is supported by NSF postdoctoral fellowship DMS-1303910.  T.R. acknowledges the support of the Academy of Finland project no. 137528.}


\renewcommand{\subjclassname}{%
 \textup{2010} Mathematics Subject Classification}

\date{August 25, 2015}

\begin{abstract}  
We show that the Heisenberg group is not minimal in looking down.
This answers 
Problem 11.15 in \emph{Fractured fractals and broken dreams} by David and Semmes, or equivalently,
Question 22 and hence also Question 24 in \emph{Thirty-three yes or no questions about mappings, measures, and metrics}
by Heinonen and Semmes.

The non-minimality of the Heisenberg group is shown
by giving an example of an Ahlfors $4$-regular metric space $X$ having big pieces of itself 
such that no Lipschitz map from a subset of $X$
to the Heisenberg group has image with positive measure,
and by providing a Lipschitz map from the Heisenberg group to the space $X$ having as image the whole $X$.

As part of proving the above result we define a new distance on the Heisenberg group that 
is bounded by the Carnot-Carath\'eodory distance, that preserves the Ahlfors-regularity,
and such that the Carnot-Carath\'eodory distance and the new distance are biLipschitz equivalent on no set of positive measure.
This construction works more generally in any Ahlfors-regular metric space where one can make suitable shortcuts.
Such spaces include for example all snowflaked Ahlfors-regular metric spaces.
With the same techniques we also provide an example of a left-invariant distance on the Heisenberg group biLipschitz to the Carnot-Carath\'eodory distance for which no blow-up admits nontrivial
dilations.
\end{abstract}

\maketitle
\newpage

\tableofcontents

\newpage
\section{Introduction}   

In \cite{davsem} David and Semmes proposed a concept of BPI  
(big pieces of itself) spaces as a notion of rough self-similarity for metric spaces. 
The definition of a BPI space requires any two balls of the space to contain big pieces 
that are biLipschitz equivalent, see Definition~\ref{def:BPI} for the precise definition. 
Self-similar fractals and Carnot groups are easy examples of BPI spaces. 
David and Semmes also introduced \emph{BPI equivalence} and a partial order for 
BPI spaces 
called \emph{looking down}. Both of them will be defined in Section~\ref{sec:preli}. 
Two BPI spaces are BPI equivalent if 
large parts of the two spaces
are
biLipschitz equivalent. A BPI metric space $X$ looks down on another BPI metric space $Y$ 
if 
$X$ and $Y$ have same Hausdorff dimension and
there is a closed subset of $X$ that can be mapped to a set of positive measure in $Y$ via a 
Lipschtiz map. 
 BPI equivalence of spaces $X$ and $Y$ implies that $X$ and $Y$ are
\emph{look-down equivalent}, meaning that $X$ looks down on $Y$ and $Y$ looks down on $X$. 
However, Laakso has shown that the converse is not true in general \cite{laakso-bpi}.

The partial ordering of BPI spaces raises the interesting question of what are the possible 
minimal spaces in this ordering. A space $X$ is \emph{mimimal} in looking down if every space
$Y$ on which $X$ looks down is look-down equivalent to $X$. For example, from the result of
Kirchheim \cite{Kirchheim} we know that Euclidean spaces are minimal in looking down. 
A quantitative version of Kirchheim's theorem was later given in \cite{Schul_BiLipschitz_Decomposition}
in which it was shown that if a map $f : [0,1]^n \to X$ has positive Hausdorff $n$-content, then it
has a quantitatively large biLipschitz piece.

David and Semmes asked in Problem 11.15 of \cite{davsem} if the Heisenberg group $\mathbb H$ is also minimal 
in looking down, when equipped with sub-Riemannian distances, also called  Carnot-Carath\'eodory distances. This was also asked as Question 22 of \cite{Heinonen_Semmes}.
We show that this is not the case.

\begin{thm}\label{thm:Hnotminimal}
 The subRiemannian Heisenberg group 
  is not minimal in looking down.
\end{thm}

This theorem has important implications in the development of a theory of rectifiability based on the Heisenberg group.
Recall that a metric measure space $(X,d,\mu)$ is countably $n$-rectifiable if there exist a countable set of
Borel subsets $A_i \subseteq \R^n$ and Lipschitz maps $f : A_i \to X$ 
such that $\mu\left(X \backslash \bigcup_i f(A_i) \right) = 0$ and $\mu \ll \H^n$ where $\H^n$ is the Hausdorff $n$-measure. 
It was shown in \cite{Kirchheim} that, by further countably decomposing each $f(A_i)$ if necessary, 
one may assume that each $f_i$ is biLipschitz.  

One can easily create a definition of being $\mathbb H$-rectifiable by letting each $A_i$
be a Borel subset of the Heisenberg group $\mathbb H$ and setting $n = 4$, the Hausdorff dimension of $\mathbb H$. 
However, we now see that there exists a metric measure space $(X,d,\mu)$ with positive Hausdorff 
4-measure that is the Lipschitz image of a subset of $\mathbb H$ but is not the countable union
of biLipschitz images of subsets of $\mathbb H$. 
Thus, ``Lipschitz rectifiability'' is strictly weaker than ``biLipschitz rectifiability'' when using the Heisenberg geometry.

Using the self-similarity of the Carnot-Carath\'eodory distance $d_{cc}$ it is easy to construct
BPI spaces that can be realized as subsets of $\mathbb H$ with self-similar type modifications of the distance $d_{cc}$.
A critical part in the proof of Theorem~\ref{thm:Hnotminimal} is to modify the distance $d_{cc}$ to get a new distance $d$ 
in such a way that with the $d_{cc}$ distance the space looks down on the space equipped with the distance $d$, but not the other way.
Such distance is constructed using a shortening technique that has been also used in \cite{LeDonne7, LeDonne_Rigot_Heisenberg_BCP} 
to give examples of distances not satisfying the Besicovitch Covering Property.
The result obtained here with the shortening  technique is the following.


\begin{thm}\label{thm:Heisenbergdistance}
Let $(\mathbb H, d_{cc})$ be the subRiemannian Heisenberg group.
There exists a distance $d$ on $\mathbb H $ such that
\begin{enumerate}
\item $d\leq d_{cc}$;
\item $(\mathbb H, d)$ is Ahlfors 4-regular;
\item if $A\subseteq \mathbb H$ is a subset with $\H^4_{cc}(A)>0$, then 
$d$ and $d_{cc}$ are not biLipschitz equivalent on $A$.
\end{enumerate}
\end{thm}

Thus, we construct an Ahlfors 4-regular metric space $X$ onto which $(\mathbb H, d_{cc})$ Lipschitz surjects, but for which this surjection has no biLipschitz pieces.  Theorem~\ref{thm:Heisenbergdistance} answers Question 24 of \cite{Heinonen_Semmes} negatively (although 
the same negative answer is provided by the
 negative answer to Question 22 given by Theorem~\ref{thm:Hnotminimal}).
 
It should be noted that this behavior changes when one requires that the target $X$ is another Carnot group.  
Indeed, one can then use a similar argument as in \cite{Kirchheim},
with inspiration from \cite{pauls},
 to show that Lipschitz maps from the Heisenberg group to another Carnot group with positive 4-measure image have biLipschitz pieces.  This statement can also be made quantitative as was done in \cite{Meyerson_Lipschitz_biLipschitz,Li-lip-bilip}.

%

Another situation where Lipschitz maps have biLipschitz pieces is when the spaces are
Ahlfors regular, linearly locally contractible topological manifolds and the target has manifold weak tangents,
see the work of G.C. David \cite{gcdavid} (this David is not the same David of David-Semmes).
We note that in Theorem \ref{thm:Heisenbergdistance} the constructed space $(\mathbb H, d)$ neither has manifold tangents nor is linearly locally
contractible.

The construction of the distance $d$ in Theorem~\ref{thm:Heisenbergdistance}
relies on the fact that in the Heisenberg group we can shorten the distance between two points
that differ only in the vertical component without affecting the distances far away from the two points.
By taking this property as an assumption we obtain a more general result.


\begin{thm}\label{thm:general}
Let $(X,\rho)$ be a metric space and $Q>0$.
Assume 
\begin{enumerate}
\item $(X,\rho)$ is Ahlfors $Q$-regular;
 \item there exists $\lambda\in(0,1)$ such that
 for all $p\in X$ and all $0 < r < \diam(X)$ there exist $q_1, q_2\in B_\rho (p,r)$ such that
 $$\rho(q_1, q_2) \geq 
 \lambda r$$
 and
 \begin{equation}\label{e:general-metric-assumption}
 \rho(p_1, p_2) \leq \rho (p_1, q_1) + \rho (p_2, q_2), \qquad \forall p_1, p_2 \notin B_\rho (p,r).  
 \end{equation} 
\end{enumerate}
Then there exists a distance $d$ on $X $ such that
\begin{enumerate}
\item $d\leq \rho$;
\item $(X, d)$ is Ahlfors $Q$-regular;
\item if $A\subseteq X$ is a subset with $\H^Q_{\rho}(A)>0$, then 
$d$ and $\rho$ are not biLipschitz equivalent on $A$.
\end{enumerate}
\end{thm}
  
We will first prove Theorem~\ref{thm:general} in Section~\ref{sec:breakingbad}.
After having proven Theorem~\ref{thm:general}, the proof of Theorem~\ref{thm:Heisenbergdistance} 
follows by showing that there is a metric on $\mathbb H$, biLipschitz equivalent to the Carnot-Carath\'eodory metric, 
that satisfies \eqref{e:general-metric-assumption}.
This will be done in Section~\ref{sec:existence}. Theorem~\ref{thm:Hnotminimal}
will then be proven in Section~\ref{sec:BPI}.
Other examples of spaces satisfying the condition in Theorem~\ref{thm:general}
are snowflakes of Ahlfors-regular metric spaces, e.g., 
the real line equipped with the square root of the Euclidean distance, see Theorem~\ref{thm:Ahlforssnowflake}.  

In the second part of the paper we consider distances on $\mathbb H$ 
that have extra homogeneity structure. For example, we  assume that left translations are biLipschitz.
We show that with the assumptions of  
Theorem~\ref{thm:Heisenbergdistance} 
such distances are locally biLipschitz equivalent to the distance $ d_{cc}$.

\begin{thm}\label{thm:positiveinvarianttobilip}
Let $d$ be a 
distance on the Heisenberg group
$\mathbb H$ such that $d \le d_{cc}$ and 
$\H^4_{d}(B_{cc}(0,1))> 0$.
Assume that the left translations in $\mathbb H$ 
are biLipschitz with respect to $d$.
Then $d$ and $ d_{cc}$ are biLipschitz equivalent on compact sets.
\end{thm}

We remark that the assumptions in Theorem~\ref{thm:positiveinvarianttobilip} are necessary.
Indeed, if we don't assume  $d \le d_{cc}$, then as a counterexample
one can take two sub-Riemannian distances on $\mathbb H$ that have two different  horizontal bundles.
If we don't assume  
$\H^4_{d}(B_{cc}(0,1))> 0$, then a counterexample is given by every Riemannian left-invariant distance.
Moreover, the distance $\min\{1, d_{cc}\}$ shows that the conclusion of the theorem may not be global.



We conclude the paper by showing that for distances that are biLipschitz equivalent to $d_{cc}$
the metric differentiation does not hold in general. 
Kirchheim's result in 
 \cite{Kirchheim}
can be stated as the fact that 
every semi-distance $d$ in $\R^n$ that is smaller than the Euclidean distance
is metrically differentiable, i.e., 
 at almost every point its blow-up is a homogeneous semi-distance.
Similarly, by \cite{Pauls01}, 
we know  that on Carnot groups semi-distances  smaller than $ d_{cc}$ 
 are metrically differentiable
   but only in the horizontal directions.  
Regarding non-horizontal directions,  from \cite{Kirchheim_Magnani} 
we know that there is
a distance 
 in the Heisenberg group that is
a counterexample to metric differentiability, although it is not biLipschitz to $d_{cc}$.
As the last result of this paper we give in Section~\ref{sec:countermetricdif} another pair of
counterexamples to metric differentiability that are biLipschitz equivalent to $ d_{cc}$
 and whose blow-ups even fail  self-similarity, 
which is a weaker property than homogeneity.
If $\{\delta_\lambda\}_{\lambda>0}$ denotes the standard one-paramenter family of isomorphisms of 
$\mathbb H $, see Section \ref{sec:Heis},
a
(semi-)distance $d$ is {\em self-similar}
 if there exists 
some $\lambda>1$ for which $d(\delta_\lambda (p),\delta_\lambda (q)) =\lambda d(p,q)$, for all $p,q\in \mathbb H $.
In the following result, by {\em a blow-up} of a distance $d$
we mean any pointwise limit of the functions 
\[(p_1,p_2)\mapsto \dfrac{1}{\lambda_j} d( q_j \delta_{\lambda_j} ( p_1), q_j  \delta_{\lambda_j} ( p_2)) ,\]
as 
$\lambda_j \to 0$ and $q_j\in \mathbb H $.

\begin{thm}[Failure of Kirchheim-metric differentiation for biLipschitz maps]\label{thm:fail}
There exist two distances $d_1, d_2$ on $\mathbb H $ that are biLipschitz equivalent to $d_{cc}$ such that
\begin{enumerate}
\item The distance $d_1$ is left-invariant, but no blow-up of $d_1$ is self-similar.
\item No blow-up of $d_2$ is left-invariant nor self-similar.
\end{enumerate}
\end{thm}
 Both Theorem~\ref{thm:positiveinvarianttobilip} and 
Theorem \ref{thm:fail} are proved in Section \ref{sec:BiLipschitz}.

\section{Preliminaries}\label{sec:preli}

We begin by recalling the definition of Hausdorff measures on a metric space $(X,d)$.  Let $Q > 0$.  Then for $A \subseteq X$, one defines
\[
 \H^Q_d(A) := \lim_{s \to 0^+} \inf \left\{ \sum_{i=1}^\infty (\diam E_i)^Q : A \subseteq \bigcup_i E_i \text{ an open cover}, \diam(E_i) < s\right\}.
\]
 We say that $\H^Q$ is the Hausdorff $Q$-measure of $(X,d)$.  It is known that the Hausdorff $Q$-measure is Borel regular although it may not be locally finite.

Let $Q > 0$.  Recall that a metric measure space $(X,d,\mu)$ is said to be Ahlfors $Q$-regular if there exists $C \geq 1$ so that
  \[\frac{1}{C} r^Q \leq \mu(B(x,r)) \leq C r^Q, \qquad \forall x \in X, r < \frac{\diam X}{2}.\]
We remark that if $(X,d,\mu)$ is Ahlfors $Q$-regular, then 
 $(X,d,\H^Q)$ is Ahlfors $Q$-regular.

A biLipschitz map $f$ between metric spaces $(X,d)$ and $(X',d')$ is said to be
$C$-{\em conformally biLipschitz} with scale factor $\lambda > 0$ if $f$ is $C$-biLipschitz between
the metric spaces $(X, \lambda d)$ and $(X',d')$.
Another term, coming from
Banach space theory,
 for the same notion is {\em quasi-similarity}.

\begin{definition}[BPI space]\label{def:BPI}
An Ahlfors $Q$-regular metric space $(X,d)$ is said to be a \emph{BPI} (``big pieces of itself'') {\em space} if there exist
constants $C \ge 1$ and $\theta > 0$ such that for all $x_1,x_2 \in X$ and
$0 < r_1 , r_2 < \diam(X)$ there is a closed set $A \subseteq B(x_1,r_1)$ with $\mathcal{H}^Q(A) \ge \theta r_1^Q$
and if there is a $C$-conformally biLipschitz embedding $f \colon A \to B(x_2,r_2)$ with scale factor $r_2/r_1$.
\end{definition}

\begin{definition}[BPI equivalence] 
Two BPI spaces $(X,d)$ and $(X',d')$ of the same dimension $Q$ are called \emph{BPI equivalent}
if there exist constants $\theta>0$ and $C>0$ such that for each $x \in X$, $x' \in X'$ and radii
$0 < R < \diam(X)$, $0 < R' < \diam(X')$ there exist a subset $A \subset B(x,R) \subset X$
with $\H^Q_d(A)\ge \theta R^Q$ and an $C$-conformally biLipschitz
embedding $f \colon A \to B(x',R')$ with scale factor $R'/R$.
\end{definition}

\begin{definition}[Looking down] 
Let $(X,d)$ and $(X',d')$ be BPI metric spaces of Hausdorff dimension $Q$.
The space $(X,d)$ is said to \emph{look down on} $(X',d')$ if there is a closed set $A \subset X$
and a Lipschitz map $f \colon A \to X'$ such that $f(A)$ has positive Hausdorff $Q$-measure. If
also $X'$ looks down on $X$, then $X$ and $X'$ are called \emph{look-down equivalent}.
\end{definition}


\subsection{The Heisenberg group and its distances}
\label{sec:Heis}

The Heisenberg group $\mathbb H$ is the simply connected Lie group whose Lie algebra is generated by three vectors $X,Y,Z$ with only non-zero relation $[X,Y]=Z$.
Via exponential coordinates it can be identified as the manifold $\R^3$ equipped with Lie multiplication:
\[  p\cdot q = \left( x_p + x_q, y_p + y_q, z_p + z_q + \frac12 (x_py_q - y_px_q) \right). \]
It follows easily from the definition that the origin $(0,0,0) \in \mathbb H$ is the identity element and that the center of the group is
\[ {\rm Z}({\mathbb H}) = \{(0,0,z) : z \in \R\}. \]
For each $\lambda > 0$, the Heisenberg group has an automorphism defined as
\begin{equation}
\label{hom:dil}
\delta_\lambda (x,y,z) := (\lambda x, \lambda y,   \lambda^2   z).
\end{equation}
A left-invariant (semi-)distance $d$ is {\em homogeneous}, with respect to \eqref{hom:dil}, if for all $\lambda>0$
\begin{equation}
\label{eq:homog:dist}
d(\delta_\lambda (p),\delta_\lambda (q)) =\lambda d(p,q), \qquad \forall p,q\in \mathbb H.
\end{equation}

Our main example of homogeneous distance is the following.
We introduce the {\em box norm}
 \[
 \norm{p} := \max\left\{|x_p|,|y_p|, \sqrt{|z_p|}\right\}.
 \]
 We define the  {\em box distance} as 
 \[
    d_b(p,q) := \norm{p^{-1}q}.
 \]
Clearly, $d_b$ is left-invariant and it satisfies \eqref{eq:homog:dist}.
To check that it satisfies the triangle inequality we need to show that 
 \[
 \norm{p\cdot q} \leq  \norm{p} + \norm{q}. 
 \] 
 First,
 \[  |x_{p\cdot q}|
=
|x_p +x_q|
\leq 
 |x_p| +|x_q|
 \leq 
 \norm{p} + \norm{q},
 \]
 and analogously for the $y$ component.
 Second,
  \begin{eqnarray*}
 \sqrt{   |z_{p\cdot q}|}
&=&
\sqrt{\left|z_p +z_q +\dfrac{1}{2} (x_p y_q - x_q y_p) \right|}\\
&\leq &
\sqrt{|z_p| +|z_q| +|x_p| |y_q| + |x_q| | y_p |}
\\&\leq& 
\sqrt{  \norm{p}^2 + \norm{q}^2 +2  \norm{p}  \norm{q} } 
\\& \leq &
 \norm{p} + \norm{q}.
 \end{eqnarray*}
Explicitely, the box distance is 
 \begin{equation}\label{eq:box:dist}
  d_b(p_1,p_2) = \max\left\{|x_1-x_2|,|y_1-y_2|, \sqrt{\left|z_1-z_2-\frac12(x_1y_2-x_2y_1)\right|}\right\}.
 \end{equation}


 Given a homomorphism $L : {\mathbb H} \to {\mathbb H}$, one can define the Jacobian to be
 \[
   J(L) = \frac{\H^4_{d_b}(L(B_{d_b}(0,1)))}{\H_{d_b}^4(B_{d_b}(0,1))}.
 \]
 Let $f : ({\mathbb H},d_b) \to ({\mathbb H},d_b)$ be a Lipschitz map.  Pansu proved in \cite{Pansu} that for almost every $x \in {\mathbb H}$ there exists a Lipschitz homomorphism $Df(x) : {\mathbb H} \to {\mathbb H}$ (the Pansu-derivative of $f$ at $x$) so that
 \[
   Df(x) (g) = \lim_{\lambda \to 0} \delta_{1/\lambda} (f(x)^{-1} f(x \delta_\lambda(g))).
 \]
 This result was extended to Lipschitz maps whose domains are measurable subsets $A \subseteq \mathbb H$ by Magnani in \cite{magnani-area}.  Magnani also used the Pansu-derivative in conjunction with the Jacobian to get the following area formula:
 \begin{equation}
  \int_{\mathbb H} N(f,A,y) d\mathcal{H}^4_{d_b}(y) = \int_A J(Df(x))\,d\mathcal{H}^4_{d_b}(x). \label{e:area-formula}
 \end{equation}
 Here, $N(f,A,y)$ is the multiplicity of $f$ with respect to the set $A$.
 
\subsection{Shortening distances}

Given a metric space $(X,\rho)$, a symmetric function $c:X\times X \to [0,\infty)$ such that $c\leq \rho$ will be called a {\em cost function}.
We denote by $\S$ all those pairs of points $(x,y)\in X\times X$ such that $c(x,y)< \rho(x,y)$
$$\S:=\{ (x,y)\in X\times X\;:\; c(x,y)< \rho(x,y)\}.$$
An element in $\S$ will be called
 {\em shortcut} (or flight or tunnel).
If we have $N\in \N$ and $x_0,x_1,\ldots,x_N\in X$ then the $N$-tuple  ${\bf x} = (x_0,x_1,\ldots,x_N)$
will be called an {\em itinerary} from the extreme points $x_0$ to $x_N$ and we set ${\rm Ext} ({\bf x}):=(x_0,x_N)$
and $\ell({\bf x}) := N$.
We will denote by $\I$ the collection of all itineraries in $X$, i.e.,
$$\I:=\{(x_0,x_1,\ldots,x_N) \;:\; N\in \N, \,x_j\in X   \}.$$
The {\em cost} of an itinerary ${\bf x}= (x_0,x_1,\ldots,x_N) \in \I$ is
$$c ({\bf x}):= \sum_{i=1}^N c(x_{i-1}, x_i).$$
The distance $d$ associated to the cost function $c$ is defined as
\begin{equation}\label{eq:ddef}
d(x,y): = \inf\{ c ({\bf x}) \;:\; {\bf x} \in \I ,\, {\rm Ext} ({\bf x})=(x,y)\}.
\end{equation}

\begin{remark} \label{r:d'-leq-c}
It is not too hard to verify symmetry and the triangle inequality for $d$ and so $d$ is a semi-distance on $X$.  If there is another distance $d'$ on $X$ such that
\[
  d'(x,y) \leq c(x,y), \qquad \forall x,y \in X,
\]
then by the triangle inequality for $d'$, we also have that
\[
  d'(x,y) \leq d(x,y), \qquad \forall x,y \in X,
\]
and so $d$ is then a distance.
\end{remark}

We define the subset of {\it alternating itineraries} as
$$\I_A:=\{{\bf x} \in \I \;:\;  (x_{j-1},x_{j})\in \S \iff j \text{ even}    \}.$$
Colloquially speaking, for each of these itineraries, one walks at every odd step and flies at every even step. Note that we allow for the 
stationary walks, i.e., the itinerary can have   $x_{j-1} = x_{j}$, for some $j$ odd.

We shall assign to each shortcut a natural number that we call {\em level} of the shortcut.
Namely, a function $L \colon \S \to \N$ will be called  a {\em level function}. 
Larger levels will usually indicate shortcuts over smaller distances.  We can also define the level function of an alternating itinerary ${\bf x} = (x_0,\dots,x_N) \in \I_A$ as the function
\begin{eqnarray*}
  L_{\bf x} : \{1,\dots,\lfloor N/2 \rfloor \} &\to& \N \\
  k &\mapsto& L(x_{2k-1},x_{2k}).
\end{eqnarray*}

We say that a function $f : \{1,\dots,n\} \to \R$ is {\em decreasing-increasing} if there is some $k \in \{1,\dots,n\}$ for which $f|_{[1,k]}$ is decreasing and $f|_{[k+1,n]}$ is increasing (both not necessarily strictly monotonically).  We can then define a further subset of itineraries with decreasing-increasing level functions:
$$\I^*:=\{{\bf x} \in \I_A \;:\;  L_\x \text{ is decreasing-increasing}    \}.$$

\section{Breaking biLipschitz equivalence using shortcuts}\label{sec:breakingbad}

In this section we prove Theorem~\ref{thm:general}. Let $(X,\rho)$ and $\lambda$ be as in the assumptions of Theorem~\ref{thm:general}.

\subsection{Constructing the shortcuts}
\label{sec:construction}
Let $(\alpha_n)_{n=1}^\infty$ be a sequence of real numbers in $(0,1)$ such that $\alpha_n \downarrow 0$. 
The number $\alpha_n$ will be the ratio of the cost of the level $n$ shortcut compared to the original distance of the shortcut.

Let us define the shortcuts one level at a time.
We define inductively the level $n$ shortcuts $\S_n \subset X \times X$, for $n \in \N$ as follows.
We set $c_E \ge 4$ to be a constant that we now fix. 
Let $\mathcal{N}_n := \{x_i\}$ be a set of points in
\[
 X \setminus \bigcup_{j=1}^{n-1}\bigcup_{(x,y) \in \S_j} B_\rho(\{x,y\},
4\lambda^n)
\]
such that $\rho(x_i,x_j) \geq 4 \lambda^n$ and $X \subseteq \bigcup_i B_\rho(x_i,c_E \lambda^n)$. 
It may be that one could choose $c_E$ so that no such $\mathcal{N}_n$ exists.
We show later in Lemma~\ref{lma:constantsexist} that there is always a choice
of the constants $\lambda$ and $c_E$ for which the set $\mathcal{N}_n$ exists.

Using assumption (2) of Theorem~\ref{thm:general} we select for each $i$ points $q_{i,1}, q_{i,2} \in B_\rho(x_i,\lambda^n)$ such that
\[
\rho(q_{i,1}, q_{i,2}) \ge \lambda^{n+1} 
\]
 and
 \begin{equation}\label{e:general-metric-assumption2}
 \rho(p_1, p_2) \le \rho(p_1,q_{i,1}) + \rho(p_2, q_{i,2}) \qquad \text{for all }p_1,p_2 \notin B_\rho(x_i,\lambda^n).
\end{equation}
Now define the level $n$ shortcuts as
\[
 \S_n := \{(q_{i,1}, q_{i,2})\,:\, i\} \cup \{(q_{i,2}, q_{i,1})\,:\, i\},
\]
their corresponding costs as
\[
 c(q_{i,1}, q_{i,2}) := c(q_{i,2}, q_{i,1}) := \alpha_n\rho(q_{i,1}, q_{i,2})
\]
and their level as
 \[
L(q_{i,1}, q_{i,2}) :=  L(q_{i,2}, q_{i,1}) := n.
\]
Finally, let
\[
 \S := \bigcup_{n=1}^\infty \S_n
\]
and define $d$ as in \eqref{eq:ddef}.

We now prove the existence of the sets $\mathcal{N}_n$ for certain choices of $\lambda$ and $c_E$.

\begin{lemma}\label{lma:constantsexist}
  There exists some $\lambda_0 \in (0,1/4)$ depending only on the Ahlfors regularity of
  $(X,\rho)$ such that if we set $\lambda \leq \lambda_0$ 
  and $c_E = 8 + \frac1\lambda$, 
  then we can always find $\mathcal{N}_n$.
\end{lemma}

\begin{proof}
  Let $\mu$ be a measure on $(X,\rho)$ so that $(X,\rho,\mu)$ is Ahlfors regular (one could use $\mu = \H^Q$ for instance).
  We may suppose by taking $\lambda$ small enough (as we are free to do) and using Ahlfors regularity of $(X,\rho,\mu)$ that
  \begin{equation}
    \mu(B_\rho(x,r/4)) - \mu(B_\rho(y,\lambda r)) - \mu(B_\rho(z, \lambda r)) > 0, \qquad \forall x,y,z \in X, 0 < r < \diam(X).\label{e:small-lambda}
  \end{equation}

  Let $A = \bigcup_{j < n} \bigcup_{(x,y) \in \S_j} \{x,y\}$.  
 By the definition of $A$, each $x \in A$ comes with a pair $x' \in A$ such that $(x,x') \in S_l$ for some $l < n$.
 We claim that 
 \begin{equation}
   \rho(x,y) \geq 2 \lambda^{n-1}, \qquad \forall y \in A \setminus \{x,x'\}. \label{e:A-spread-out}
 \end{equation}
 To see this, taking $y \in A \setminus \{x,x'\}$, there exists $y' \in A$ such that $(y,y') \in S_k$ for some $k < n$.
 Let $x_i^l \in \mathcal{N}_l$ such that $x,x' \in B_\rho(x_i^l,\lambda^l)$ and $x_j^k \in \mathcal{N}_k$ such that 
 $y,y' \in B_\rho(x_j^k,\lambda^k)$.
We consider two cases. Suppose  first that $k = l$. By the $4\lambda^k$ separation of $\mathcal{N}_k$ we then have
 \[
  \rho(x,y) \ge \rho(x_i^l,x_j^k) - \rho(x,x_i^l) - \rho(y,x_j^k) \ge 4\lambda^k - \lambda^k - \lambda^k = 2\lambda^k \ge 2\lambda^{n-1}.
 \]
 Suppose now that $k \ne l$. By symmetry we may assume $k < l$. Then by construction, $x_i^l \notin B_\rho(y,4\lambda^l)$ and
  thus
 \[
  \rho(x,y) \ge \rho(x_i^l,y) - \rho(x,x_i^l) \ge 
  4\lambda^l - \lambda^l = 3\lambda^l 
  \ge 2\lambda^{n-1}.
 \]
 Thus \eqref{e:A-spread-out} is proven.
  
  Let $\{x_i\}$ be a maximal $4\lambda^n$-separated net of $X \setminus B_\rho(A,4\lambda^n)$.
  Let $x \in X$.  Suppose there exists $y \in A$ such that $\rho(x,y) < 4\lambda^n$.
  As $\lambda < 1/4$, we get by \eqref{e:A-spread-out} that the number of balls $\{B_\rho(p,4 \lambda^n)\}_{p \in A}$ 
  that intersect $B_\rho(y,\lambda^{n-1})$ is at most 2.  This, together with \eqref{e:small-lambda}, gives that
  \[
    B_\rho(y,\lambda^{n-1}) \setminus B_\rho(A,4\lambda^n) \neq \emptyset.
  \]
  Thus, there exists some $z \in B_\rho(y,\lambda^{n-1}) \setminus B_\rho(A,4\lambda^n)$.
  As $\{x_i\}$ is also a $4\lambda^n$ covering of $X \setminus B_\rho(A,4\lambda^n)$, 
  we get that there exists some $x_i$ such that $\rho(z,x_i) < 4\lambda^n$.  Altogether, we get that
  \[
    \rho(x,x_i) \leq \rho(x,y) + \rho(y,z) + \rho(z,x_i) < 4\lambda^n + \lambda^{n-1} + 4 \lambda^n 
    =  \left(8 +  \frac{1}{\lambda} \right)\lambda^n.
  \]  
  In the case when $x \notin B_\rho(A,4\lambda^n)$, we are also done as the set  $\{x_i\}$ is a $4\lambda^n$-cover of $X \setminus B_\rho(A,4\lambda^n)$.
\end{proof}

\subsection{Properties of the new distance}
In this section we point out some properties of the distance $d$, for example, the fact that it is a distance.
We start by showing that $d$ can be equivalently given by infimizing costs over itineraries with decreasing-increasing level functions.
We first show that, if the level function 
on an alternating itinerary
is not decreasing-increasing, then there exists a shorter alternating itinerary with the same endpoints of 
no greater cost.

\begin{lemma}\label{lma:weak-id}
  Suppose $\x = (x_0,\dots,x_N) \in \I_A$ and there exists $j \in 2\N-1$ such that
  $$L(x_{j+2},x_{j+3}) \geq \max(L(x_j,x_{j+1}),L(x_{j+4},x_{j+5})).$$
  Then the itinerary ${\bf x}' = (x_0',\dots,x_{N-2}') \in \I_A$ where
  \[
    x_k' = \begin{cases}
      x_k & k \in \{0,\dots,j+1\}, \\
      x_{k+2}  & k \in \{j+2,\dots,N-2\}, 
    \end{cases}
  \]
  satisfies $\Ext({\bf x}) = \Ext({\bf x}')$ and $c({\bf x}') \leq c({\bf x})$.
\end{lemma}

\begin{proof}
  That $\Ext({\bf x}) = \Ext({\bf x}')$ is obvious from construction.  
  Consider the subitinerary ${\bf y} = (x_{j+1},x_{j+2},x_{j+3},x_{j+4})$.  
  We claim that $c(x_{j+1},x_{j+4}) \leq c({\bf y})$, which proves the lemma.
  Let $x \in \mathcal{N}_n$ be the point for which the shortcut $(x_{j+2},x_{j+3})$ was found in $B_\rho(x,\lambda^n)$.
  Then 
  $n= L (  x_{j+2},x_{j+3}  )$. 
  We claim that  
  $x_{j+1},x_{j+4} \notin B_\rho(x,\lambda^n)$.  Indeed, by assumption, $L(x_j,x_{j+1}) \leq n$.
  So first suppose $L(x_j,x_{j+1}) < n$.  Then
  $x$ was found in the complement of
  \[
    B(x_j,2\lambda^n) \cup B(x_{j+1},2\lambda^n).
  \]
  If instead $L(x_j,x_{j+1}) = n$, then let  $y \in \mathcal{N}_n$ be the point for which 
  the shortcut $(x_j,x_{j+1})$ was found
  in $B_\rho(y,\lambda^n)$.
  We may assume that $x\neq y$, otherwise $\{x_j,x_{j+1}\}=\{x_{j+2},x_{j+3}\}$ and the claim is obvious.
Hence,   we have that $\rho(x,y) \geq 4 \lambda^n$ and so
  \[
    B_\rho(x,\lambda^n) \cap B_\rho(y,\lambda^n) = \emptyset.
  \]
  As $x_{j+1} \in B_\rho(y,\lambda^n)$, we get that $x_{j+1} \notin B_\rho(x,\lambda^n)$.
  A similar argument holds for $x_{j+4}$.
  Thus, we have that the claim follows from 
  \eqref{e:general-metric-assumption2} if we set $x_i=x$, $q_{i,1} = x_{j+2}$, $q_{i,2} = x_{j+3}$, $p_1 = x_{j+1}$, and $p_2 = x_{j+4}$.
\end{proof}

\begin{lemma}\label{lma:multiplicity}
  For any $\x \in \I_A$, there exists $\x' \in \I^*$ such that $\Ext(\x) = \Ext(\x')$, $c(\x') \leq c(\x)$, and
  \begin{equation}
    \# L_{\x'}^{-1}(k) \leq 4, \qquad \forall k \in \N. \label{e:4-bound}
  \end{equation}
  Moreover, $\x$ and $\x'$ have the same first and last shortcuts.
\end{lemma}

\begin{proof}
  Given an initial $\x \in \I_A$, we iterate Lemma~\ref{lma:weak-id} until we get an itinerary $\x'$ 
  for which there are no indices that satisfy the hypothesis of Lemma~\ref{lma:weak-id}. 
  As the length of the itinerary shrinks by 2 with each application of Lemma~\ref{lma:weak-id},
  we get that we have to stop after some finite number of iterations.  
  It is elementary to see that if $L_{\bf x'} : \{0,\dots,n\} \to \N$ satisfies
  $$L_{\x'}(i+1) < \max (L_{\x'}(i),L_{\x'}(i+2)), \qquad \forall i \in \{0,\dots,n-2\},$$
  then $L_{\x'}$ is decreasing-increasing, which means that $\x' \in \I^*$.

  Now suppose $\# L_{\x'}^{-1}(k) > 4$ for some $k \in \N$.  As $\x' \in \I^*$, we get that there exists some $i \in 2\N$ so that
  $$L(x_i,x_{i+1}) = L(x_{i+2},x_{i+3}) = L(x_{i+4},x_{i+5}).$$
  But this contradicts the assumption that $\x'$ does not have any indices that satisfy the hypothesis of Lemma~\ref{lma:weak-id}.
  
  Finally, since each application of Lemma~\ref{lma:weak-id} keeps the first and last shortcut of $\x$ unchanged the
  resulting itinerary $\x'$ has the same first and last shortcut as $\x$.
\end{proof}

\begin{proposition}
 The function $d$ is a distance on $X$.
\end{proposition}

\begin{proof}
 The validity of the triangle inequality follows from the definition of the distance as defined in \eqref{eq:ddef}. Symmetry is due to the 
 symmetry of the cost function. What needs to be checked is that $x \ne y$ implies $d(x,y)>0$. 
 In order to show this, suppose that $x,y \in X$ with $\rho(x,y)>0$.
 Let $n \in \N$ be such that 
 \[
  8 \frac{\lambda^{n}}{1-\lambda} \le \frac12\rho(x,y).
 \]
 Let $(\alpha_n)$ be the sequence of positive numbers used to construct the cost function in Section~\ref{sec:construction}. Consider the positive number
 \[
  \varepsilon := \min\left(\frac12\min_{k \in [1,n-1]} \alpha_k\lambda^{k+1}, \frac14\rho(x,y)\right).
 \]
 Let $\x = (x_0,\dots,x_N) \in \I^*$ with 
 $\Ext(\x) = (x,y)$, $c(\x) \leq d(x,y) + \varepsilon$,
 and $\# L_{\x}^{-1}(k) \leq 4$ for all $k \in \N$, which
exists by Lemma~\ref{lma:multiplicity} (remember that using 
stationary walks
every itinerary can be modified to be an alternating itinerary of no greater cost, because of triangle inequality). 
 
 On the one hand, if $L_\x^{-1}([1,n-1]) = \emptyset$, then the   alternating itinerary does not have shortcuts 
 at odd steps and it has them at
  even steps only of level greater that $n$ and with multiplicity at most $4$. Hence, we get
 \begin{eqnarray*}
   d(x,y) & \ge & c(\x) - \varepsilon \ge \sum_{j \text{ odd}}\rho(x_{j-1},x_j) - \varepsilon \ge \rho(x,y) - \sum_{j \text{ even}}\rho(x_{j-1},x_j) - \frac14\rho(x,y)\\
  & \ge & \frac34\rho(x,y) - 4 \sum_{k=n}^\infty 2\lambda^k \ge \frac34\rho(x,y) - 8 \frac{\lambda^{n}}{1-\lambda} \ge \frac14\rho(x,y) > 0,
 \end{eqnarray*}
 where we used that a point in a shortcuts at level $k$ has $\rho$-distance less than $  \lambda^k$ from the center of the ball in which the shortcut was found. 
 On the other hand, if $L_\x^{-1}([1,n-1]) \ne \emptyset$, then,
 if $(x_{\ell-1},x_\ell)$ is a shortcut at level $l<n$ of $\x$, we have 
 \[
  d(x,y) \ge c(\x) - \varepsilon \ge 
 c(x_{\ell-1},x_\ell) - \frac12\min_{k \in [1,n-1]} \alpha_k\lambda^{k+1}
  \ge
  \frac12\min_{k \in [1,n-1]} \alpha_k\lambda^{k+1} > 0,
 \]
 where we used that
 $ c(x_{\ell-1},x_\ell) = \alpha_l \rho(x_{\ell-1},x_\ell) \geq \alpha_l \lambda^l.$
 In both cases $d(x,y) > 0$ as needed.
\end{proof}

\begin{lemma}\label{lma:largegaps}
 Let $x \in X$ and $0 < r < \lambda^n$.  There exists at most one pair $\{q_1,q_2\}$ such that
 $(q_1,q_2) \in \S$, $L(q_1,q_2) < n$ and $\{q_1,q_2\} \cap B_d(x,r) \ne \emptyset$.
\end{lemma}

\begin{proof}
 Suppose to the contrary that there exist two disjoint $(q_1,q_2), (\tilde q_1,\tilde q_2) \in \S$ with
 \begin{equation}
  L(q_1,q_2), L(\tilde q_1,\tilde q_2)  < n \label{e:shortcut-lower-bnds}
 \end{equation}
 and $q_1,\tilde q_1 \in B_d(x,r)$. Then
\[
  d(q_1,\tilde q_1) \le d(q_1,x) + d(x,\tilde q_1) < 2r \le 2\lambda^n.
 \]
 Let $\x = (x_0,\dots,x_N) \in \I^*$ with $x_0=q_1,$ $x_N=\tilde q_1$ and
 \begin{equation}
  c(\x) < 2\lambda^n. \label{eq:q1q1small}
 \end{equation} 
We may assume that $N$ is odd, up to adding a stationary walk at the end.
Hence the slightly longer itinerary ${\bf y}=(q_2,q_2,q_1,x_1,\dots,x_{N-1}, \tilde q_1,\tilde q_2)$ is an alternating itinerary.
 By applying Lemma~\ref{lma:multiplicity} to ${\bf y}$
   we know that there exists 
  $  {\bf y}'  \in \I^*$ 
  such that
  $
   c( {\bf y}') \le c( {\bf y}),
  $
  where 
$  {\rm Ext}( {\bf y}')  ={\rm Ext} ( {\bf y}) $.
Notice that, since the construction in the proof of  
Lemma~\ref{lma:multiplicity}
keeps the
  the first and last shortcuts stay the same,
  we have that ${\bf y}'$ is of the form
  \[{\bf y}' = (q_2,q_2,q_1,x'_1,\dots,x'_{N'}, \tilde q_1,\tilde q_2)
   .\]
    We conclude that 
    the itinerary $\x$ may 
    be replaced with no extra cost by an itinerary $\x'$ that is decreasing-increasing and, due 
    to \eqref{e:shortcut-lower-bnds}, 
    we have $L_\x^{-1}([n,\infty]) = \emptyset$.


We remark that
the itinerary $\x'$ cannot have only stationary walks, i.e., 
 there is some $j \in 2\Z$ so that $x'_j \neq x'_{j+1}$. Indeed,  otherwise
 the itinerary
  cannot move away from $\{q_1,q_2\}$, since distinct  shortcuts are separated. 

Hence, there are two distinct shortcuts
$(x'_{j-1},x'_j)$,
$(x'_{j+1},x'_{j+2})$.
Set 
$k_1:= L(x'_{j-1},x'_j)$
$k_2:= L(x'_{j+1},x'_{j+2})$. Recall that $k_1,k_2<n$.
Let $a, b$ are the centers 
  of the balls in which the shortcuts $(x'_{j-1},x'_j)$,
$(x'_{j+1},x'_{j+2})$ were found with radii $\lambda^{k_1}$ and $\lambda^{k_2}$, respectively.
Let us distinguish two cases. 
Assume first that $k_1=k_2=:k$, so that $a$ and $ b$ are $4\lambda^{k}$ separated. Hence, we have
    \[
 \rho(x'_j,x'_{j+1}) \ge 
   \rho(a,b) -\rho(a,x'_j) -\rho(b,x'_{j+1}) 
   \ge  4 \lambda^{k} - \lambda^{k}- \lambda^{k}
    \ge
    2\lambda^{k} 
    \ge 2\lambda^n.
  \]
  Suppose now $k_1\neq k_2$, say that $k_1< k_2$, the other case is similar.
Recall that $b$ was found outside $B(x'_j,4\lambda^{k_2})$ in the construction of the shortcuts.
Hence, we have
    \[
 \rho(x'_j,x'_{j+1}) \ge 
    \rho(b,x'_j) -\rho(b,x'_{j+1}) 
   \ge  4 \lambda^{k_2} - \lambda^{k_2} 
    \ge
    3\lambda^{k_2} 
    \ge 3\lambda^n.
  \]
  In either case we have
    \[
   c(\x) \ge  c(\x') \ge \rho(x'_j,x'_{j+1}) \ge 
      2\lambda^n,
  \]
which is    in contradiction with \eqref{eq:q1q1small}.
\end{proof}

Next lemma will be used for the proof of the Ahlfors $Q$-regularity in the next section.
\begin{lemma}\label{lma:inclusions}
 For all $x \in X$ and $r > 0$ there exist $y_1,y_2 \in X$ such that
 \begin{equation}\label{eq:inclusions}
  B_\rho(x,r) \subseteq B_d(x,r)  \subseteq B_\rho\left(\{y_1,y_2\},\left(2 + 8/(\lambda -\lambda^2)\right)r\right).
 \end{equation}
\end{lemma}

\begin{proof}
 The first inclusion $B_\rho(x,r) \subseteq B_d(x,r)$ follows from the fact that by construction $d \le \rho$.
 
 Let us show the second inclusion.  Suppose first that $r \ge 1$. Let $z \in B_d(x,r)$.
 By Lemma~\ref{lma:multiplicity} there exists $\x = (x_0,\dots,x_N) \in \I^*$ with $\Ext(\x)=(x,z)$, $c(\x) \leq r$, and $\# L_\x^{-1}(k) < 4$ for all $k \in \N$. Then
 \begin{eqnarray*}
  \rho(x,z) & \le & \sum_{j=1}^N\rho(x_{j-1},x_j) \le
   \sum_{j\text{ odd}}\rho(x_{j-1},x_j) + \sum_{j\text{ even}}\rho(x_{j-1},x_j)\\
   & \le &  c({\bf x}) + 8\sum_{k = 1}^\infty \lambda^k \leq r + 8\frac{\lambda}{1-\lambda} \le  \left(2 + 8\frac{\lambda^{-1}}{1-\lambda} \right)r,
 \end{eqnarray*}
 since $\lambda <1<r$, and hence \eqref{eq:inclusions} holds with $y_1=y_2=x$.
 
 Now suppose that $r < 1$ and let $n \in \N \cup\{0\}$ be such that
 \[
   \lambda^{n+1} \le r < \lambda^n.
 \]
 By Lemma~\ref{lma:largegaps} there exists at most one pair $\{y_1,y_2\}$ such that $(y_1,y_2) \in \S$ and
 \begin{equation}\label{eq:onlyonebigjump}
   L(y_1,y_2) < n \quad\text{and}\qquad B_d(x,r) \cap \{y_1,y_2\} \ne \emptyset.
 \end{equation}
 If such pair $\{y_1,y_2\}$ does not exist, we define $y_1 = y_2 = x$.
 Now 
 \[
  B_d(x,r) \subset B_d(\{y_1,y_2\},2r).
 \]
 Take $z \in B_d(x,r)$. By symmetry we may suppose $d(z,y_1) \le d(z,y_2)$.  By Lemma~\ref{lma:multiplicity}
 there exists 
 ${\bf x} = (x_0,\dots,x_N) \in \I^*$ with ${\rm Ext} ({\bf x})=(y_1,z)$, 
 $c({\bf x}) < 2r$, and $\# L_\x^{-1}(k) < 4$ for all $k \in \N$.
 By the assumption $d(z,y_1) \le d(z,y_2)$ we may assume that ${\bf x}$ does not contain
 the shortcuts $(y_1,y_2),(y_2,y_1)$.
 Then by \eqref{eq:onlyonebigjump} we have $L(x_{j-1},x_j) \ge n$ for all $j$ even. 
 Then, by the fact that ${\bf x} \in \I^*$, we have
 \begin{eqnarray*}
   \rho(y_1,z) & \le & \sum_{j=1}^N\rho(x_{j-1},x_j) \le
  \sum_{j\text{ odd}}\rho(x_{j-1},x_j) + \sum_{j\text{ even}}\rho(x_{j-1},x_j)\\
     & \le & c({\bf x}) + 8\sum_{k = n}^\infty \lambda^k = 2r + 8\frac{\lambda^n}{1-\lambda}
     \le  \left(2 + 8\frac{\lambda^{-1}}{1-\lambda} \right)r,
 \end{eqnarray*}
 since $ \lambda^{n+1} \le r$, and hence \eqref{eq:inclusions} holds.
\end{proof}

\subsection{Ahlfors $Q$-regularity of $(X,d)$}
We now give the proof of the Ahlfors $Q$-regularity of the space $(X,d)$, assuming that  
  $(X,\rho)$ is Ahlfors $Q$-regular.
  Namely, we have that there exists a constant $C < \infty$ such that
\[
 \frac{1}{C}r^Q \le \mathcal{H}^Q_\rho(B_\rho(x,r)) \le Cr^Q,
\]
for all $x \in X$ and $0 < r < \diam_\rho(X)$.

Hence by Lemma~\ref{lma:inclusions} we have
\[
 \frac{1}{C}r^Q  \le  \mathcal{H}^Q_\rho(B_d(x,r))  \le C2(2 + 8/(\lambda -\lambda^2))^Qr^Q,
\]
for all $x \in X$ and $0 < (2 + 8/(\lambda -\lambda^2))r < \diam_\rho(X)$. Thus $(X,d)$ is also Ahlfors $Q$-regular.

\subsection{No biLipschitz pieces}

Let $A\subseteq X$ be such that $\H^4_{\rho}(A)>0$. 
Our aim is to show that $d$ and $\rho$ are not biLipschitz equivalent on $A$.
For this purpose take a density-point $x$ of $A$. Then for any $\epsilon>0$ there exists $r_\epsilon > 0$
such that
\[
 B_\rho(x,r) \subset B_\rho(A,\epsilon r), \qquad \text{for all }r \in (0,r_\epsilon).
\]
Now, for all $n \in \N$ there exists $(q_{n,1},q_{n,2}) \in \S$ with $L(q_{n,1},q_{n,2}) = n$ such that
\[
 \{q_{n,1},q_{n,2}\} \subset B_\rho(x,2c_E\lambda^n).
\]
If $3c_E\lambda^n < r_\epsilon$, there exist $x_{n,1},x_{n,2} \in A$ such that 
\[
\rho(x_{n,1},q_{n,1}) \le 3c_E\epsilon\lambda^n
\qquad \text{and}\qquad
\rho(x_{n,2},q_{n,2}) \le 3c_E\epsilon\lambda^n.
\]
Then
\[
 \rho(x_{n,1},x_{n,2}) \ge \rho(q_{n,1},q_{n,2}) - \rho(x_{n,1},q_{n,1}) - \rho(x_{n,2},q_{n,2})
 \ge \lambda^{n+1} - 6c_E\epsilon\lambda^n
\]
and
\begin{eqnarray*}
 d(x_{n,1},x_{n,2}) & \le & d(q_{n,1},q_{n,2}) + d(x_{n,1},q_{n,1}) + d(x_{n,2},q_{n,2})\\
         & \le & \alpha_n \rho(q_{n,1},q_{n,2}) + \rho(x_{n,1},q_{n,1}) + \rho(x_{n,2},q_{n,2})\\
         & \le & 2\alpha_n\lambda^n + 6c_E\epsilon \lambda^n.
\end{eqnarray*}
Therefore we have
\[
 \frac{d(x_{n,1},x_{n,2})}{\rho(x_{n,1},x_{n,2})} \le \frac{2\alpha_n\lambda^n + 6c_E\epsilon \lambda^n}{\lambda^{n+1} - 6c_E\epsilon\lambda^n} = \frac{2\alpha_n + 6c_E\epsilon}{\lambda-6c_E\epsilon}.
\]
As $\alpha_n \to 0$, by letting $n$ be sufficiently large and $\epsilon$ be sufficiently small, we get that $\alpha_n + 6c_E\epsilon$ is sufficently small and so
the distances $d$ and $\rho$ are not biLipschitz equivalent on $A$.

\begin{remark}\label{rmk:zerocost}
 In the definition of the costs for the shortcuts we could also allow $\alpha_n = 0$.  This would give a semi-distance on $X$ that, upon factoring, gives a metric space $(Y,d)$ that satisfies all the previous properties.
\end{remark}

\section{Existence of shortcuts}\label{sec:existence}

We will now verify that the shortcuts necessary to employ Theorem \ref{thm:general}
can be made in the subRiemannian Heisenberg group and in any snowflaked Ahlfors regular metric space.

\subsection{Shortcuts in the Heisenberg group}

\begin{proof}[Proof of Theorem~\ref{thm:Heisenbergdistance}]
 We will verify that the assumptions of Theorem~\ref{thm:general} hold in the Heisenberg group with $\lambda = \frac12$.
 Let $p \in \mathbb H$ and $r > 0$. By left-translation invariance of the distance $d_b$
 in $\mathbb H$ we may assume that $p = (0,0,0)$.
 Take $q_1 = (0,0,0)$ and $q_2 = (0,0,r^2/4)$. Now let $p_1,p_2 \notin B(0,r)$.
 Since $d_b(q_1,q_2) = \sqrt{r^2/4} = r/2$, by the triangle inequality we have that $d_b(p_1,q_1) \ge r/2$
 and $d_b(p_2,q_2) \ge r/2$. Write $p_1 = (x_1,y_1,z_1)$ and $p_2 = (x_2,y_2,z_2)$.
 Then
the equation for the box distance is given by \eqref{eq:box:dist}.
 Trivially,  we have 
 \[
  |x_1-x_2| \le |x_1| + |x_2| \le d_b(p_1,q_1) + d_b(p_2,q_2)
 \]
 and
 \[
  |y_2-y_1| \le |y_1| + |y_2| \le d_b(p_1,q_1) + d_b(p_2,q_2).
 \]
 By using the triangle inequality and the estimate $r^2/4 \le d_b(q_1,p_1)d_b(q_2,p_2)$ we also get
\begin{eqnarray*}
   \left|z_1-z_2-\frac12(x_1y_2-x_2y_1) \right| & \le & |z_1|+|z_2-\frac14r^2|+\frac14r^2 +\frac12|x_1||y_2| + \frac12|x_2||y_1|\\
   & \le & d_b(p_1,q_1)^2+d_b(p_2,q_2)^2+d_b(q_1,p_1)d_b(q_2,p_2)\\
   & & +\frac12d_b(q_1,p_1)d_b(q_2,p_2)+\frac12d_b(q_1,p_1)d_b(q_2,p_2)\\
   & \le & d_b(p_1,q_1)^2+2d_b(q_1,p_1)d_b(q_2,p_2)+d_b(p_2,q_2)^2\\
   & = & \left(d_b(q_1,p_1) + d_b(q_2,p_2)\right)^2.
 \end{eqnarray*}
 Thus we have 
 \[
  d_b(p_1,p_2) \le d_b(q_1,p_1) + d_b(q_2,p_2)
 \]
 as required by the assumptions of Theorem~\ref{thm:general}.
\end{proof}

\subsection{Shortcuts in snowflaked Ahlfors regular metric spaces}

\begin{theorem}\label{thm:Ahlforssnowflake}
 Let $(X,d)$ be an Ahlfors $Q$-regular metric space with $Q > 0$ and let $\delta \in (0,1)$.
 Then the snowflaked metric space $(X,d^\delta)$ satisfies the assumptions of Theorem~\ref{thm:general}.
 Consequently, there exists a distance $d'$ on $X$ such that $d' \le d^\delta$, $(X,d')$ is Ahlfors $Q/\delta$-regular,
 and for any $A\subseteq X$ with $\H^{Q/\delta}_{d^\delta}(A)>0$, we have that 
 $d'$ and $d^\delta$ are not biLipschitz equivalent on $A$.
\end{theorem}

\begin{proof}
First of all, it is trivial that $(X,d^\delta)$ is Ahlfors $Q/\delta$-regular.
Let us then check the assumption (2) of Theorem~\ref{thm:general}. 
Since $(X,d)$ is $Q$-regular, there exists $C > 1$ such that
\begin{equation}\label{eq:ahlfors}
 \frac{1}{C}r^Q \le \mathcal{H}^Q(B_d(x,r)) \le Cr^Q
\end{equation}
for all $x \in X$ and $0 < r < \diam(X)$.
We shall set $\lambda:= (2C)^{-2\delta/Q}(1-\delta)^\delta$.
Take $p \in X$ and $0 < r < \diam(X)$. Define $q_1 = p$ and take
\[
 q_2 \in B_d(p,(1-\delta)r^\frac{1}{\delta}) \setminus B_d(p,
 (2C)^{-2/Q}(1-\delta)
 r^\frac{1}{\delta}).
\]
Such $q_2$ exists since the annulus from where the point is taken has positive measure by \eqref{eq:ahlfors}
and is hence non-empty. In particular, $q_1, q_2\in B_{d^\delta} (p,r)$ and
\[
 d(q_1,q_2)^\delta \ge \left((2C)^{-2/Q}(1-\delta)\right)^\delta r = \lambda r.
\]
Now, take $p_1,p_2 \notin B_d\left(p,r^\frac{1}{\delta}\right)$.  We get that
$$ d(q_1,p_1),d(q_2,p_2)  \geq \delta r^{1/\delta},$$
and so
$$d(q_1,q_2) \leq \left( 1 - \delta\right) r^{1/\delta} \leq \left( \frac{1}{\delta} - 1 \right) \min(d(q_1,p_1),d(q_2,p_2)).$$
First assume that $d(p_1,q_1) \le d(p_2,q_2)$.
Then we get
\begin{eqnarray*}
 d(p_1,p_2)^\delta & \le & (d(p_1,q_1)+d(q_1,q_2)+d(p_2,q_2))^\delta \\
 & \le & ((1+(\frac1\delta-1))d(p_1,q_1)+d(p_2,q_2))^\delta\\
 & \le & d(p_2,q_2)^\delta + \delta(1+(\frac1\delta-1))d(p_1,q_1)d(p_2,q_2)^{1-\delta} \\
 & \le & d(p_1,q_1)^\delta + d(p_2,q_2)^\delta
\end{eqnarray*}
verifying \eqref{e:general-metric-assumption}.  In the penultimate inequality, we used the fact that $x \mapsto x^\delta$ is concave so that the higher order terms of the Taylor expansion is always negative.  An analogous calculation takes care of the case $d(p_2,q_2) \leq d(p_1,q_1)$.
\end{proof}

\section{A BPI space using self-similar shortcuts in the Heisenberg group}\label{sec:BPI}

 In this section we prove Theorem~\ref{thm:Hnotminimal}. 
 The idea is to consider a regular subset $K\subset \mathbb H$, to specify in a self-similar way the shortcuts
 taken in the construction of Section~\ref{sec:breakingbad}
 and to make all the shortcuts to have zero cost. This way the similitude mappings used in the selection of shortcuts will
 almost be similitude
 mappings also for the new distance $d$. This will allow us to show that $(K,d)$ is BPI.
 Then the facts that $(\mathbb H,d_b)$ looks down on $(K,d)$ and that
 $(K,d)$ does not look down on $(\mathbb H,d_b)$ follow, after some work, via Theorem~\ref{thm:general}.


 \subsection{Defining a self-similar tiling}

 Define the similitude mappings as
 \[
   S_{i,j,k}(p) = \left(\frac{i}{2},\frac{j}{2},\frac{k}{4}\right) \cdot \delta_\frac12(p), \qquad i,j \in \{0,1\}, k \in \{0,1,2,3\}.
 \]
 Relabel the similitudes  by $\{S_i\,:\,i = 1,\dots,16\} = \{S_{i,j,k}\}$ and 
 denote by $K$ the attractor of $\{S_{i,j,k}\}$, i.e., the nonempty compact set  (see \cite{Hutchinson} for details)
 satisfying
 \begin{equation}
   K = \bigcup_{i=1}^{16}S_i(K). \label{e:K-self-similar}
 \end{equation}
 Let us show that $K$ has nonempty interior. 
 First of all, the horizontal projection of the iterated function system has the unit square as the attractor.
 Secondly, since the dilation and the group operation commute, we may consider separately the horizontal and vertical components
 of the iterated function system. This way we see that
 \begin{equation}\label{eq:Kformula}
  K = \{(x,y,z+t) \,:\, (x,y,z) \in \widetilde K, t \in [0,1]\},
 \end{equation}
 where $\widetilde K$ is the attractor of the horizontal component that is realized as the attractor of the system
 \[
  \{S_{i,j,0}\,:\, i,j \in \{0,1\}\}.
 \]
 The set $\widetilde K$ has the form
 \begin{equation}\label{eq:phidef}
   \widetilde K = \overline{\{(x,y,\varphi(x,y))\,:\,(x,y)\in [0,1]^2\}}
 \end{equation}
 with some Borel function $\varphi \colon [0,1]^2 \to \R$. Observe that $\varphi$ is bounded since $K$ is compact.
 Also, since $0$ is the fixed point of $S_{0,0,0}$ and $S_{i,j,0}(\widetilde K)$ do not contain $0$ if $i\ne 0$ or $j \ne 0$, 
 the function $\varphi$ is continuous at $0$.
 Therefore by \eqref{eq:Kformula} the attractor $K$ contains a small ball near $0$ and thus $K$ has nonempty interior.
 Because of the nonempty interior and the self-similar structure $(K,d_b)$ is Ahlfors $4$-regular.
 

 \subsection{Constructing the shortcuts}
For a multi-index  
$\mathtt{i}=
 (i_1,...,i_k)  \in \{1,\dots,16\}^k$, we shall use the standard notation
$S_{\mathtt{i}}$ for the composition \[ 
S_{\mathtt{i}}:=  S_{i_1} \circ S_{i_2} \circ \cdots \circ S_{i_k}.\]
With $k=0$ we interpret $\{1,\dots,16\}^k$ to consist of only one element, call it $\emptyset$, and $S_\emptyset$ is then understood
to be the identity map.
 We define the shortcuts at level $n$ as 
 \[
   \S_n := \{(S_{\mathtt{i}}(\frac12,\frac12,0),S_{\mathtt{i}}(\frac12,\frac12,\frac1{64}))\,:\, \mathtt{i} \in \{1,\dots,16\}^{n-3}\}.
 \]
 We also set $L(x,y) = n$ for $(x,y) \in \S_n$.  Note that levels start from $n = 3$.
 We then define the total set of shortcuts as
 \[
   \S := \bigcup_{n=3}^\infty \S_n.
 \]
 Define the cost as $c(x,y) = 0$ for all $(x,y) \in \S$.
 Let us check that the construction of Section~\ref{sec:breakingbad} works with this choice of shortcuts and costs.
 This will be established by the following three lemmas for 
 $\lambda = 1/2$ and $c_E \geq 4$
 some sufficiently large number.


 In the following lemmas we will use the map
  \[\pi : ({\mathbb H},d_b) \to \R^2\]
  that is  the projection onto the $xy$-plane 
   and is a $1$-Lipschitz homomorphism, when we endow $\R^2$ with the $\ell_\infty$-distance. 
 The first lemma shows that the points near which we find the level $n$ shortcuts can be found outside a $4\lambda^n$-neighborhood shortcut points of lower levels.

 \begin{lemma}
   For all $n \geq 3$, we have
   \[
    \{S_{\mathtt{i}}(\frac12,\frac12,0) \,:\, \mathtt{i} \in \{1,\dots,16\}^{n-3}\} \cap \left( \bigcup_{m < n} \bigcup_{(x,y) \in \S_m} B_{d_b}(\{x,y\},2^{-n+2}) \right) = \emptyset.
   \]
 \end{lemma}

 \begin{proof}
   Let $A_k \subset [0,1]^2$ be the centers of the dyadic subcubes of level $k$.  Note that for each $k$ we have that
   \[
     \pi(\{S_{\mathtt{i}}(\frac12,\frac12,t) \,:\, \mathtt{i} \in \{1,\dots,16\}^k, t \in \R\}) = A_k.
   \]
   As $\pi$ is 1-Lipschitz, it suffices to prove that
   \[
     A_n \cap B_{\R^2_\infty} \left( \bigcup_{m < n} A_m, 2^{-n-1} \right) = \emptyset.
   \]
   But this follows from the geometry of $(\R^2,\|\cdot\|_\infty)$.  Note that we need the sets $B_{\R^2_\infty}$ to be open, which is fine.
 \end{proof}

 The next lemma says that the points where we find the level $n$ shortcuts themselves are $4\lambda^n$-separated.
 
 \begin{lemma}
  For all $n \geq 3$ the set
  \[
A:=   \{S_{\mathtt{i}}(\frac12,\frac12,0) \,:\, \mathtt{i} \in \{1,\dots,16\}^{n-3}\}
  \]
  is $2^{-n+2}$-separated (note: $2^{-n+2} = 4 \cdot 2^{-n}$, which is needed for the construction).
 \end{lemma}

 \begin{proof}
   As shown in the previous lemma, the image of 
$A$
   under $\pi$ is precisely the centers of the dyadic subcubes of $[0,1]$ of level $n-3$.  Let $x,y \in A$ and suppose $\pi(x) \neq \pi(y)$.  Then
   \[
     d_b(x,y) \geq \|\pi(x) - \pi(y)\| \geq 2^{-n+3}.
   \]
   Now suppose $\pi(x) = \pi(y)$ but $x \neq y$.  Then as the vertical component of the iterated function system can be viewed independently, we see that the $z$-coordinate of $x$ and $y$ are points in the center of the level $n-3$ 4-dic subintervals of $[\varphi(\pi(x)),\varphi(\pi(x))+1]$,
   where $\varphi$ is the function in \eqref{eq:phidef}. 
    Thus, they differ by no less than $4^{-n+3}$ and so
   \[
     d_b(x,y) \geq \sqrt{4^{-n+3}} = 2^{-n+3}>2^{-n+2} .
   \]
 \end{proof}

 Finally, we show that the level $n$ shortcut points form a $c_E\lambda^n$-covering of $K$ for sufficiently large $c_E$.  This finishes all the properties needed to construct the shortcuts.
 
 \begin{lemma} \label{lma:K-c1}
  There exists some absolute constant $c_E > 0$ so that
  \begin{equation}\label{e:K-ball-ind}
   K \subseteq \bigcup_{\mathtt{i} \in \{1,\dots,16\}^{n-3}} B_{d_b}(S_{\mathtt{i}}(\frac12,\frac12,0),c_E 2^{-n}), \qquad \forall n \geq 3.
  \end{equation}
 \end{lemma}

 \begin{proof}
  We prove the claim by induction.
  As $K$ is bounded, we easily get \eqref{e:K-ball-ind} for $n=3$ by choosing some $c_E$ large enough. 
  Now assume that \eqref{e:K-ball-ind} holds for some $n \ge 3$.
  Then by the self-similarity of $K$ as exhibited in \eqref{e:K-self-similar} we get
  \begin{eqnarray*}
    K&  = & \bigcup_{i=1}^{16} S_i(K) \subseteq \bigcup_{i=1}^{16} S_i(\bigcup_{\mathtt{i} \in \{1,\dots,16\}^{n-3}} B_{d_b}(S_{\mathtt{i}}(\frac12,\frac12,0),c_E 2^{-n}))\\
    & = &\bigcup_{\mathtt{i} \in \{1,\dots,16\}^{n-2}} B_{d_b}(S_{\mathtt{i}}(\frac12,\frac12,0),c_E 2^{-n-1})
  \end{eqnarray*}
  Thus \eqref{e:K-ball-ind} holds for $n+1$.
 \end{proof}
 
 By Remark~\ref{rmk:zerocost} taking zero costs for shortcuts is allowed.
 From the proof of Theorem~\ref{thm:Heisenbergdistance} we see that $\lambda = \frac12$ works in the Heisenberg group. 
 
 Now the conclusions of Theorem~\ref{thm:general} hold for the constructed distance $d$.
 That is, the identity map $\id \colon (K,d_b) \to (K,d)$ is Lipschitz, 
 but not biLipschitz on any set of positive measure, and the space $(K,d)$ is Ahlfors regular.
 In particular, $(\mathbb H,d_b)$ looks down on $(K,d)$. In order to show that $(\mathbb H,d_b)$ is not minimal in 
 looking down, we still need to prove that $(K,d)$ is a BPI space and that
 $(K,d)$ does not look down on $(\mathbb H,d_b)$.

 \subsection{$(K,d)$ is a BPI space}
  Note that Lemma~\ref{lma:multiplicity} holds in $(K,d)$ as its proof, as well as the proof of Lemma~\ref{lma:weak-id}, only require that $\mathcal{N}_{n+1} \cap B_{d_b}(\mathcal{N}_n,4 \lambda^n) = \emptyset$, which holds for 
  the construction of $(K,d)$.  We begin with the following lemma.
 
  \begin{lemma}\label{lma:small-walks}
    Let $\x \in \I^*$ and $n = \min \{k \in \N : L_\x^{-1}(k) \neq \emptyset\}$.  There exists $\x' \in \I^*$ such that $\Ext(\x) = \Ext(\x')$, $c(\x') \leq c(\x)$, and for any $j \in \N$ such that $\min (L(x_{2j-1},x_{2j}),L(x_{2j+1},x_{2j+2})) = m < n$, then
    \[
      d_b(x_{2j},x_{2j+1}) \leq 2 \lambda^m.
    \]
  \end{lemma}

  \begin{proof}
    We may suppose without loss of generality that there exist $k \in L^{-1}(n)$ such that $k \geq 2j+1$, that is, $L_\x$ is still decreasing from $2j-1$ to $2j+1$.  Thus, $L(x_{2j+1},x_{2j+2}) = m$.  
    
    Let $x \in \mathcal{N}_m$ be the point for which the shortcut $(x_{2j+1},x_{2j+2})$ is found in $B_{d_b}(x,\lambda^m)$.  Then as $L(x_{2j+3},x_{2j+4}) \leq m$, we get that $x_{2j+3} \notin B(x, \lambda^m)$.  If $d_b(x_{2j},x_{2j+1}) > 2\lambda^m$, then we get that $x_{2j} \notin B(x,\lambda^m)$ by the triangle inequality.  Thus, applying \eqref{e:general-metric-assumption} with $q_1 = x_{2j+1}$, $q_2 = x_{2j+2}$, $p_1 = x_{2j}$, and $p_2 = x_{2j+3}$, we get that we can replace $(x_{2j},x_{2j+1},x_{2j+2},x_{2j+3})$ in $\x$ with $(x_{2j},x_{2j+3})$ to get a itinerary in $\I^*$ with lower cost and two fewer points with the same extremal points.
    
    We then iterate this procedure until we cannot to get our needed itinerary $\x'$.
  \end{proof}
 
 The following lemma says that one can connect $x,y \in S_{\mathtt{i}}(K)$ by an itinerary that does not go too far out.

In this section we write $|\mathtt{i}|=k$ if $\mathtt{i} \in \{1,\dots,16\}^k$.

 \begin{lemma}
   There exists some $C > 0$ so that 
 for all multi-indices ${\mathtt{i}}$,
    for all $x,y \in S_{\mathtt{i}}(K)$ and all $\epsilon > 0$, there exists an itinerary $\x = (x_0,\dots,x_n) \in \I^*$ such that $c(\x) \leq (1+\epsilon) d(x,y)$, $\Ext(\x) = (x,y)$, and $x_0,\dots,x_n\in B_{d_b}(S_{\mathtt{i}}(K),C 2^{-|\mathtt{i}|})$.
 \end{lemma}

 \begin{proof}
   We claim that there exists some constant $M \in \N$ depending only on $c_E > 0$ of Lemma~\ref{lma:K-c1} such that if $\epsilon > 0$, $\mathtt{i} \in \{1,\dots,16\}^k$ for $k \geq M$, and $x,y \in S_{\mathtt{i}}(K)$, then there exists some itinerary $\x \in \I^*$ such that
   \begin{enumerate}
    \item $c(\x) \leq (1+\epsilon) d(x,y)$,
    \item $\Ext(\x) = (x,y)$,
    \item $\# L_{\x}^{-1}(k) \leq 4, \qquad \forall k \in \N$,
    \item $d_b(x_{j-1},x_j) \leq 2^{1-m}$ for all $j$ odd such that
      \[
        m = \min(L(x_{j-2},x_{j-1}),L(x_j,x_{j+1})) < \min \{k : L_\x^{-1}(k) \neq \emptyset\},
      \]
    \item $d_b(x_{j-1},x_j) \leq 2^{M/2-k}$ for all $j$ odd,
    \item $L_{\x}^{-1}([0,k-M]) = \emptyset$.
   \end{enumerate}
   If the claim holds, then we get that for all $x,y \in S_{\mathtt{i}}(K)$ with  $\mathtt{i} \in \{1,\dots,16\}^k$ for $k \geq M$, there exists an itinerary $\x = (x_0,\dots,x_n) \in \I^*$ such that $c(\x) \leq (1+\epsilon) d(x,y)$, $\Ext(\x) = (x,y)$, and 
   \[
     \sum_{i=0}^{n-1} d_b(x_i,x_{i+1}) \leq C 2^{-|\mathtt{i}|}
   \]
   for some $C$ depending on $M$.  The lemma now follows for all $\mathtt{i} \in \{1,\dots,16\}^k$ with $k \geq M$ from the triangle inequality
   and for all $\mathtt{i} \in \{1,\dots,16\}^k$ with $k < M$ by the fact that there are only finitely many such $\mathtt{i}$.

   Let us prove the claim.  Let $c_E > 0$ be the constant from Lemma~\ref{lma:K-c1}.
   Let $M \in \N$ be the minimal even number such that
   \begin{equation}
    2^{M/2} > 2c_E \label{e:M/2-defn-1}
   \end{equation}
   and
   \begin{equation}
     2^{M-1} - 2^{M/2+1} - 2^{M/2+7} > c_E 2^{-2}. \label{e:M/2-defn-2}
   \end{equation}
   As $x,y \in S_{\mathtt{i}}(K)$,   if $|\mathtt{i}|=k>M$, we get that
   \begin{equation}
     d(x,y) \leq d_b(x,y) \leq c_E2^{-k}. \label{e:K-c1-defn}
   \end{equation}

   By an application of Lemma~\ref{lma:multiplicity} on some itinerary with cost no more than $(1+\epsilon)d(x,y)$, we get an itinerary $\x = (x_0,\dots,x_N)$ that satisfies the first three properties.  We then apply Lemma~\ref{lma:small-walks} on $\x$ (and still calling the result $\x$) to get that the fourth property is satisfied.

   Suppose $d_b(x_{j-1},x_j) > 2^{M/2-k}$ for some odd $j$.  Then
   \[
     c(\x) \geq d_b(x_{j-1},x_j) \geq 2^{M/2-k} \overset{\eqref{e:K-c1-defn} \wedge \eqref{e:M/2-defn-1}}{>} 2 d(x,y),
   \]
   a contradiction.  Thus, the fifth condition is satisfied.
   
   Now suppose $L_\x^{-1}([0,k-M]) \neq \emptyset$.  Suppose first that $\# L_{\x}^{-1}([0,k-M/2]) \geq 2$.  Then as $\x \in \I^*$, there exists some $(x_{2j},x_{2j+1}) \notin \S$ such that
   \[
     c(\x) \geq c(x_{2j},x_{2j+1}) \geq 4 \cdot 2^{M/2-k} \overset{\eqref{e:K-c1-defn} \wedge \eqref{e:M/2-defn-1}}{>} 2 d(x,y),
   \]
   which is a contradiction.

   Now suppose that $\# L_{\x}^{-1}([0,k-M/2]) = 1$.  Let $L(x_{2j-1},x_{2j}) \geq k-M$.  Then we have by the triangle inequality
   \begin{multline*}
     d_b(x,y) \geq d_b(x_{2j-1},x_{2j}) - \sum_{\ell=0}^{2j-3} d_b(x_\ell,x_{\ell+1}) - d_b(x_{2j-2},x_{2j-1}) \\
     - d_b(x_{2j},x_{2j+1}) - \sum_{\ell=2j+1}^\infty d_b(x_\ell,x_{\ell+1}) = (*).
   \end{multline*}
   We have that
   \begin{eqnarray}
     \sum_{\ell=0}^{2j-3} d_b(x_\ell,x_{\ell+1}) &\leq& 16 \sum_{s=k-M/2}^\infty 2^{-s} = 2^{M/2-k + 6} \notag \\
     \sum_{\ell=2j+1}^{\infty} d_b(x_\ell,x_{\ell+1}) &\leq& 16 \sum_{s=k-M/2}^\infty 2^{-s} = 2^{M/2-k + 6}. \label{e:small-tails}
   \end{eqnarray}
   Together with the fourth property, we get that
   \[
     (*) \overset{\eqref{e:small-tails}}{\geq} 2^{M-k-1} - 2^{M/2-k+1} - 2^{M/2-k+7} \overset{\eqref{e:M/2-defn-2}}{>} c_E 2^{-2-k}.
   \]
   But this is a contradiction because $x,y \in S_{\mathtt{i}}(K)$ and so $d_b(x,y) \leq c_E 2^{-2-k}$.
 \end{proof}

 We can now prove the following lemma that says that there exists large subset of every $S_{\mathtt{i}}(K)$ that can be connected optimally by itineraries only in $S_{\mathtt{i}}(K)$.

 \begin{lemma} \label{lma:j-defn}
   There exists some multi-index $\mathtt{j}$ such that the following property holds.  For any $\epsilon > 0$, $k \in \N$, $\mathtt{i} \in \{1,\dots,16\}^k$, and any two $x,y \in S_{\mathtt{i}}(S_{\mathtt{j}}(K))$, there exists an itinerary $\x = (x_0,\dots,x_N) \in \I^*$ with $c(\x) \leq (1+\epsilon) d(x,y)$, $\Ext(\x) = (x,y)$, and $x_0,\dots,x_N \in S_{\mathtt{i}}(K)$.
 \end{lemma}

 \begin{proof}
   Let $C > 0$ be the constant from the previous lemma.  As $K$ has nonempty interior we may choose $x \in \mathrm{int}(K)$ and $h > 0$ so that $B_{d_b}(x,h) \subset K$.  As $K$ is compact, there then exists some $\mathtt{j}$ so that
   \[
    B_{d_b}(S_{\mathtt{j}}(K),C 2^{-|\mathtt{j}|}) \subseteq B_{d_b}(x,h) \subseteq K.
   \]

   Now let $x,y \in S_{\mathtt{i}}(S_{\mathtt{j}}(K))$ for some arbitrary $\mathtt{i} \in \{1,\dots,16\}^k$.  Then there exists an itinerary $\x = (x_0,\dots,x_n) \in \I^*$ such that $c(\x) \leq (1+\epsilon) d(x,y)$, $\Ext(\x) = (x,y)$, and each of the points of $\x$ is contained
   \[
     B_{d_b}(S_{\mathtt{i}}(S_{\mathtt{j}}(K)),C 2^{-|\mathtt{i}|-|\mathtt{j}|}) = S_{\mathtt{i}}(B_{d_b}(S_{\mathtt{j}}(K),C 2^{-|\mathtt{j}|})) \subseteq S_{\mathtt{i}}(K).
   \]
 \end{proof}

 \begin{lemma}
   Let $\mathtt{j}$ be from Lemma~\ref{lma:j-defn}.  Then for all $k \in \N$ and $\mathtt{i} \in \{1,\dots,16\}^k$, we have that
   \begin{equation}
     d(S_{\mathtt{i}}(x),S_{\mathtt{i}}(y)) = 2^{-|\mathtt{i}|} d(x,y), \qquad \forall x,y \in S_{\mathtt{j}}(K). \label{e:j-invariance}
   \end{equation}
 \end{lemma}

 \begin{proof}
   If $\x$ is an itinerary from $x$ to $y$, then $S_{\mathtt{i}}(\x)$ is an itinerary from $S_{\mathtt{i}}(x)$ to $S_{\mathtt{i}}(y)$ with $c(S_{\mathtt{i}}(\x)) = 2^{-|\mathtt{i}|} c(\x)$ by the properties of the shortcuts.  Thus, we get that $d(S_{\mathtt{i}}(x),S_{\mathtt{i}}(y)) \leq 2^{-|\mathtt{i}|} d(x,y)$.

   For any $\epsilon > 0$ and $S_{\mathtt{i}}(x),S_{\mathtt{i}}(y) \in S_{\mathtt{i}}(S_{\mathtt{j}}(K))$, we get from Lemma~\ref{lma:j-defn} that there exists an itinerary $\x = (x_0,\dots,x_N)$ from $S_{\mathtt{i}}(x)$ to $S_{\mathtt{i}}(y)$ such that $c(\x) \leq (1+\epsilon) d(S_{\mathtt{i}}(x),S_{\mathtt{i}}(y))$ and $x_j \in S_{\mathtt{i}}(K)$.  Thus, applying $S_{\mathtt{i}}^{-1}$ to $\x$, we get an itinerary $\x'$ from $x$ to $y$ such that $x_j' \in K$ and $c(\x') = 2^{|\mathtt{i}|} c(\x)$.  Thus,
   \[
     d(x,y) \leq c(\x') = 2^{|\mathtt{i}|} c(\x) \leq 2^{|\mathtt{i}|} (1+\epsilon) d(S_{\mathtt{i}}(x),S_{\mathtt{i}}(y)).
   \]
   Taking $\epsilon \to 0$ then gives the lemma.
 \end{proof}

 We can now prove that $K$ is BPI.
 Let $\mathtt{j}$ be the multi-index from Lemma~\ref{lma:j-defn}, $p_1,p_2 \in K$, and $0 < r_1,r_2 < \diam(K)$.
 Now there exist two multi-indices $\mathtt{i_1},\mathtt{i_2}$ such that $\diam(S_{\mathtt{i}_j}(K)) \ge c r_j$
 and $S_{\mathtt{i}_j}(K) \subset B(p_j,r_j)$ for $j \in \{1,2\}$. Define $A = S_{\mathtt{i_1}}(S_{\mathtt{j}}(K)) \subset B(p_1,r_1)$.
 Then the map $f \colon A \to B(p_2,r_2)$ defined as $f = S_{\mathtt{i}_2}\circ S_{\mathtt{i}_1}^{-1}$ satisfies
 \[
   d(f(p),f(q)) \overset{\eqref{e:j-invariance}}{=} 2^{|\mathtt{i}_1|-|\mathtt{i}_2|}d(p,q), \qquad \forall p,q \in A.
 \]
 Since $\mathcal{H}_d^4(A) \ge c\mathcal{H}_d^4(B(p_1,r_1))$ for some $c$ depending only on $\mathtt{j}$ and $2^{|\mathtt{i}_1|-|\mathtt{i}_2|}$ is comparable to
 $r_2/r_1$, we are done with showing that $(K,d)$ is BPI.

%
%
%

 \subsection{$(K,d)$ does not look down on $(\mathbb H,d_b)$}
 By contradiction, suppose that $(K,d)$ does look down on $(\mathbb H,d_b)$. 
 Then there would exist a closed set $A \subset K$ and a 
 Lipschitz map $f \colon (A,d) \to (\mathbb H,d_b)$ with $\mathcal{H}^4(f(A))>0$. Since $d \le d_b$, also 
 $f \colon (A,d_b) \to (\mathbb H,d_b)$ is $L$-Lipschitz. Then $f$ is Pansu-differentiable almost everywhere in $A$.
 Moreover, the Pansu-differential $Df(x)$ is bijective on a set $A' \subset A$ of positive measure by the area formula:
 \[
  0 < \mathcal{H}^4(f(A)) \overset{\eqref{e:area-formula}}{\le} \int_A J(Df(x))\,d\mathcal{H}^4_{d_b}(x).
 \]
 Since for all $n,m \in \N$ the set
 \[
  B_{n,m} = K \setminus \bigcup_{k = n}^\infty \bigcup_{ \mathtt{i} \in \{1,\dots,16\}^k}S_{\mathtt{i}}(K\cap B_{d_b}(0,\frac1m))
 \]
 has $\mathcal{H}^4_{d_b}$-measure zero as a porous set, the set
 \begin{equation}\label{eq:Adef}
  A'' = A' \setminus \bigcup_{n=1}^\infty \bigcup_{m=1}^\infty B_{n,m}
 \end{equation}
 has positive measure. Let $x \in A''$ be a density point of $A''$.
 Since $x \in A''$, by the definition \eqref{eq:Adef} there exists a sequence $(n_m)_{m=1}^\infty$ of integers with
 $n_m \to \infty$ as $m \to \infty$ and a sequence of multi-indices $\mathtt{i}_m$ with $|\mathtt{i}_m| = n_m$ such that
 $x \in S_{\mathtt{i}_m}(K)$ for all $m \in \N$ and
 \[
  d_b(x,x_m)< \frac1m 2^{-{n_m}} ,
 \]
 for all $m \in \N$, where $x_m = S_{\mathtt{i}_m}(0)$.

 
 Let $K_m = \delta_{2^{n_m}}(x^{-1}A)\cap K$. Then the functions $f_m \colon (K_m,d) \to (\mathbb H,d_b)$ defined as
 \[
   f_m(p) = \delta_{2^{n_m}}(f(x)^{-1}f(x\delta_{2^{-{n_m}}}(p)))
 \]
 satisfy 
 
 \begin{equation}\label{eq:fm_estimate}
  \begin{aligned}
   d_b(f_m(p),f_m(q)) & =  d_b(\delta_{2^{n_m}}(f(x)^{-1}f(x\delta_{2^{-{n_m}}}(p))),\delta_{2^{n_m}}(f(x)^{-1}f(x\delta_{2^{-{n_m}}}(q)))) \\
   & =  2^{n_m} d_b(f(x\delta_{2^{-{n_m}}}(p)),f(x\delta_{2^{-{n_m}}}(q))) \\
   & \le  2^{n_m} L d(x\delta_{2^{-{n_m}}}(p),x\delta_{2^{-{n_m}}}(q))\\
   & \le  L d(S_{\mathtt{i}_m}^{-1}(x\delta_{2^{-{n_m}}}(p)),S_{\mathtt{i}_m}^{-1}(x\delta_{2^{-{n_m}}}(q)))\\
   & =  L d(\delta_{2^{n_m}}(x_m^{-1}x)p,\delta_{2^{n_m}}(x_m^{-1}x)q)\\
   & \le  L \left(d(\delta_{2^{n_m}}(x_m^{-1}x)p,p) + d(p,q) + d(q,\delta_{2^{n_m}}(x_m^{-1}x)q)\right),
  \end{aligned}
 \end{equation}
 where the first inequality follows from the fact that $f$ is $L$-Lipschitz and the second inequality from the fact that
 \[
   S_{\mathtt{i}_n}(\S) \subset \S.
 \]
 Notice that 
 \begin{equation}\label{eq:errorterm}
  d(\delta_{2^{n_m}}(x_m^{-1}x)p,p) \le d_b(\delta_{2^{n_m}}(x_m^{-1}x)p,p) \to 0
 \end{equation}
 as $m \to \infty$ since $d_b(\delta_{2^{n_m}}(x_m^{-1}x),0) \to 0$ as $m \to \infty$.
 The convergence in \eqref{eq:errorterm} holds uniformly for $p \in K$ by the compactness of $K$.
 
 Since $x$ is a density point of $A''$ and hence of $A$, we have that 
 for all $p \in K$ there exists a sequence $(p_m)_{m=1}^\infty$  with $p_m \in K_{n_m}$ and $d_b(p_m,p) \to 0$
 as $m \to \infty$. Along this sequence 
 by the fact that $Df(x)$ is homogeneous we get 
 \begin{eqnarray*}
   d_b(f_m(p_m),Df(x)(p_m)) & = &
   d_b(\delta_{2^{n_m}}(f(x)^{-1}f(x\delta_{2^{-{n_m}}}(p_{n_m}))),\delta_{2^{{n_m}}}(Df(x)(\delta_{2^{-{n_m}}}(p_m))))\\
   &=&2^{n_m} d_b(f(x)^{-1}f(x\delta_{2^{-{n_m}}}(p_m)),Df(x)(\delta_{2^{-{n_m}}}(p_m))) \to 0,
 \end{eqnarray*}
 as $m \to \infty$. 
 Hence also
 \begin{equation}\label{eq:fmdestimate}
  \begin{aligned}
   d_b(f_m(p_m),Df(x)(p)) & \le d_b(f_m(p_m),Df(x)(p_m))
    + d_b(Df(x)(p_m),Df(x)(p)) \to 0,
  \end{aligned}
 \end{equation}
 as $m \to \infty$.

 Combining the estimates \eqref{eq:fmdestimate}, \eqref{eq:fm_estimate} and \eqref{eq:errorterm} with the fact that 
 $d(p_m,p) \leq d_b(p_m,p) \to 0$ we get
 \begin{eqnarray*}
   d_b(Df(x)(p),Df(x)(q)) & \le & d_b(f_m(p_m),Df(x)(p)) + d_b(f_m(q_m),Df(x)(q)) \\
   & & \hspace{5.7cm}+ d_b(f_m(p_m),f_m(q_m))\\
   & \le & d_b(f_m(p_m),Df(x)(p)) + d_b(f_m(q_m),Df(x)(q)) \\
   & & \hspace{0.7cm}+ L\left(d(\delta_{2^{n_m}}(x_m^{-1}x)p_m,p_m) +  d(q_m,\delta_{2^{n_m}}(x_m^{-1}x)q_m)\right)\\
   & &\hspace{3.59cm} + L\left(d(p_m,p) + d(p,q) + d(q,q_m)\right)\\
   & \rightarrow & L d(p,q), \qquad \text{as }m \to \infty.
 \end{eqnarray*}
 Hence $Df(x)$ is also Lipschitz from $(K,d)$ to $(\mathbb H, d_b)$.
 Since $Df(x) \colon (\mathbb H, d_b) \to (\mathbb H, d_b)$ is biLipschitz, also the identity map $\id \colon (K, d)
 \to (\mathbb H, d_b)$ is Lipschitz, but we have shown that this is not the case in Theorem \ref{thm:Heisenbergdistance}.

\section{BiLipschitz equivalent distances on the Heisenberg group}
\label{sec:BiLipschitz}

In the previous sections we constructed and studied distances that were not biLipschitz equivalent on large sets.
In this final section we turn to study distances that are biLipschitz equivalent. 
First we prove Theorem \ref{thm:positiveinvarianttobilip}
showing that adding to Theorem~\ref{thm:Heisenbergdistance} the assumption that
the left-translations are biLipschitz for the distance $d$
forces the distances $d_{cc}$ and $d$ to be biLipschitz equivalent on compact sets.
After this we prove Theorem \ref{thm:fail} giving examples of distances on the Heisenberg group that are
biLipschitz equivalent with $d_{cc}$ having no self-similar tangents.

\subsection{BiLipschitz left-translations: Proof of Theorem~\ref{thm:positiveinvarianttobilip}}



Since $ d_{cc}$ is biLipschitz equivalent to the box distance, up to multiplying $d$ by a constant we assume that 
$d \le d_{b}$.

 Using the Baire Category
Theorem one can show that there exists $L>1$ such that,  if we restrict to a compact set, then the distance $d$ is $L$-biLipschitz invariant, see
 \cite[Lemma 6.7]{LeDonne1}.
 Suppose that the claim of the theorem is not true. 
 Hence, by the left-biLipschitz invariance of the distances,
 for all $N\in \N  $ there exists a point $p \in B_{d_b}(0,\frac1{2N})$
 with $d_b(0,p) > LN d(0,p)$.
 
 
 Write $r_N := d_b(0,p)>0$. We claim that we have that
 \begin{eqnarray}\label{ciccio}
     \bigcup_{n=0}^{N}B_{d_b}(p^n,\frac12 r_N)
  \subset B_d(0,2 r_N)
  . 
 \end{eqnarray}¥
 Indeed,
 if $q\in 
 B_{d_b}(p^n,\frac12 r_N)$, for some $n\leq N$, then 
\begin{eqnarray*}
d(0,q) & \leq &d(0,p) + d(p,p^2)+\ldots + d(p^{n-1} , p^n) + d(p^n , q)\\
&\leq &   LN d(0,p) + d_b(p^n , q)
\\
&\leq & d_b(0,p) + r_N/2  < 2r_N.
\end{eqnarray*}
Moreover, for $i < j<N$, we have that
 \[
   d_b(p^i,p^j) = d_b(0,p^{j-i}) \geq d_b(0,p) = r_N.
 \]

 Let $\{q_i\}_{i \in I_N}$ be a maximal $4r_N$-separated net of points with respect to distance $d$ in $B_{d_b}(0,1)$.
 First,
 by \eqref{ciccio} for all $i \in I_N$ we have that 
$
\{B_{d_b}(q_i p^n,\frac12 r_N)  \}_{n=0}^{N} $
is a disjointed collection of subsets of
$   B_{d_b}(q_i,2r_N)
 $.
 Second, 
 $\{B_d(q_i,2r_N)\}_{i \in I_N}$ is a disjointed collection of subsets of
 $B_{d_b}(0,2)$.
 Hence 
 
%
%
 
 \[
 \#I_N N \mathcal{H}_{d_b}^4(B_{d_b}(0,\frac12 r_N))
  \leq
  \mathcal{H}_{d_b}^4(B_{d_b}(0,2))
  .
 \]
Since $\{B_d(q_i,8r_N)\}_{i \in I_N}$ covers $B_{d_b}(0,1)$, by  definition of Hausdorff measure we deduce
 \begin{eqnarray*}
  \mathcal{H}_d^4(B_{d_b}(0,1)) & \le & \liminf_{N \to \infty} \#I_N (16r_N)^4\\
  &\le& \liminf_{N \to \infty} \frac{\mathcal{H}_{d_b}^4(B_{d_b}(0,2))}{N\mathcal{H}_{d_b}^4(B_{d_b}(0,\frac12 r_N))} (16r_N)^4\\
  & =&  \liminf_{N \to \infty} \frac{64^4}{N} = 0.
 \end{eqnarray*}
 This contradicts the assumption $\H^4_{d}(B_{d_b}(0,1))> 0$.
\qed

\subsection{Distances without self-similar tangents}\label{sec:countermetricdif}

In this final section we prove Theorem \ref{thm:fail}. Namely, we construct
two distances $d_1, d_2$ on $\mathbb H $ that are biLipschitz equivalent to $d_{cc}$ such that
\begin{enumerate}
\item the distance $d_1$ is left-invariant and for all $\lambda_j \to 0$ such that the distances
$$(p,q)\mapsto \dfrac{1}{\lambda_j} d_1(\delta_{\lambda_j} (p),\delta_{\lambda_j} (q)) $$
converge pointwise to some $\rho$, the distance $\rho$ is not self-similar;
\item for all $\lambda_j \to 0$ and $q_j\in \mathbb H$ such that the distances
$$(p,q)\mapsto \dfrac{1}{\lambda_j} d_2(q_j \delta_{\lambda_j} (p), q_j \delta_{\lambda_j} (q)) $$
converge pointwise to some $\rho$, the distance $\rho$ is not self-similar nor left-invariant. 
\end{enumerate}

We will first construct the distance $d_1$ and at the end indicate how the construction can be modified
to obtain the distance $d_2$.


The distance $d_1$ is defined via \eqref{eq:ddef}. The initial distance is $d_{cc}$
and the shortcuts are defined by first taking a sequence of shortcuts from the origin to points in the vertical direction
and then left-translating the shortcuts to start from every point of the space.
Since we want none of the tangents to admit nontrivial dilations, we have to be careful in defining the sequence of shortcuts.
 
 
 Let us define the set of shortcuts from the origin as
 \[
  \S_0 =  \left\{ \left(0,(0,0,4^{-n})
   \right)   \;:\;   n \in a^{-1}(\{1\})  \right\},
 \]
 where $a \colon \N \to \{0,1\}$ is a function determining whether a shortcut is taken on scale $4^{-n}$.
 If we were to take $a(n) = 1$ for all $n$, then the tangents would be self-similar.
 
 The full set of shortcuts is then defined as
 \[
  \S =
  \left\{ (p q_1, p q_2 )\; : \; p \in \mathbb H , \,(q_1,q_2) \vee (q_2, q_1) \in   \S_0
  \right\}
 \]
 and the cost function $c \colon \mathbb H \times \mathbb H \to [0,\infty)$ for $(p,q) \in \S$ as
 \[
  c(p,q) = \frac12 d_{cc}(p,q).
 \]
 The distance $d_1$ is then defined as the $d$ in \eqref{eq:ddef}.
 
 Since $\frac12 d_{cc} \le d_1 \le d_{cc}$, the function $d_1$ is a distance and it is biLipschitz equivalent with $d_{cc}$.
 By the left-invariace of the set of shortcuts $\S$, the distance $d_1$ is also left-invariant.
 
 Since we want to avoid self-similarity, we define the function $a$ so that every word written in the alphabets $\{0,1\}$
 appears consecutively in the sequence $(a(n))_{n\in\N}$ only some limited number of times.
 This is achieved for example by defining
 \[
  a(i) := \begin{cases}
          1, & \text{if there exists }k \text{ odd and }l \in \N\text{ such that }i = (k\ell\prod_{h < \ell}p_h+1)p_\ell\\
          0, & \text{otherwise},
         \end{cases}
 \]
 where $p_\ell$ is the $\ell$:th prime number.
 
 Most of the remainder of the section will be devoted to proving that with this selection of $a$ no blow-up of $d_1$ is self-similar.
 On the level of $a$ the needed property is stated in the next lemma.

 \begin{lemma}\label{lma:norepeat}
  Let $\ell \geq 1$.  There exists some $m \geq 1$ so that for any $i \geq 1$,
  there exists some $j \in \{i,i+1,\dots,i + m\ell\}$ such that $a(j) \neq a(j+m\ell)$.
\end{lemma}
\begin{proof}
 Let us write
 \[
  P_\ell = \left\{(k\ell \prod_{h < \ell}p_{h}+1)p_\ell \,:\, k \in \N\right\}.
 \]
 We claim that $\{P_\ell\}_{\ell \in \N}$ is a disjointed collection of sets. In order to see this take $0 < \ell < \ell'< \infty$
 and notice that on one hand for every $k \in \N$ we have
 $p_\ell \mid (k\ell \prod_{h < \ell}p_{h}+1)p_\ell$. On the other hand, since $p_\ell \mid \prod_{h < \ell'}p_{h}$, we have 
 $p_\ell \nmid (k\ell'\prod_{h < \ell'}p_{h}+1)p_{\ell'}$ for all $k \in \N$.
 
 Now let $\ell \geq 1$ be given. Define $m = \prod_{h \le \ell}p_{h}$.
 Then $P_\ell = \{p_\ell + m\ell k \,:\, k \in \N\}$.
 Let $i \ge 1$ and select $j \in \{i,i+1,\dots,i + m\ell\}$
 such that $j \equiv p_\ell \pmod{m\ell}$. Then by definition, $j \in P_\ell$. By the fact that the sets $P_{\ell'}$
 are pairwise disjoint we have from the definition of $a$ that
 \[
 a(m\ell k + p_{\ell}) = \begin{cases}  
                     1,& \text{if }k \text{ is odd},\\
                     0,& \text{if }k \text{ is even}.
                    \end{cases}
 \]
 Thus $a(j) \neq a(j + \ell m)$.
\end{proof}
 
 The next lemmas will be used to connect the blown up distances to the distance $d_1$, and in particular to $a$.
 
\begin{lemma} \label{lma:vertical-itinerary}
  Let $\x = (x_0,\dots,x_N)$ be an itinerary such that $x_0 = 0$ and $x_N \in {\rm Z}({\mathbb H})$.  Then there exists another itinerary ${\bf y} = (y_0,\dots,y_M)$ such that $\Ext({\bf y}) = \Ext(\x)$, $y_{i+1}^{-1}y_i \in {\rm Z}({\mathbb H})$ for all $i$, and $c({\bf y}) \leq c(\x)$.
\end{lemma}

\begin{proof}
  For the itinerary $\x = (x_0,\dots,x_N)$, we define $d_k = x_i^{-1}x_{i-1}$.  Let
  $$A = \{k : d_k \in {\rm Z}({\mathbb H})\}.$$
  We can define a bijection $\sigma : \{1,\dots,N\} \to \{1,\dots,N\}$ that maps $\{1,\dots,|A|\}$ to $A$ and preserves the ordering of $A^c$ (thus, we shift $A$ to the beginning in any order).  We now define the itinerary $(y_0,\dots,y_{|A|+1})$ where $y_0 = 0$, $y_{|A|+1} = x_N$, and $y_i = d_{\sigma(1)} \cdots d_{\sigma(i)}$.  As we only rearranged elements that are in the center, we get that $y_{|A|+1}^{-1}y_{|A|}$ is precisely the product (in order) of all the noncentral elements of $d_k$ and is itself central.
  
  It remains to show that $c({\bf y}) \leq c(\x)$.  We have that
  \begin{equation}
    c(x_{k-1},x_k) = c(y_{\sigma^{-1}(k)-1},y_{\sigma^{-1}(k)}), \qquad \forall k \in A. \label{e:A-c}
  \end{equation}
  As $d_k \notin {\rm Z}({\mathbb H})$ for $k \notin A$, we get that $(x_{k-1},x_k) \notin \S$ for $k \notin A$ and so
  \begin{equation}
    \sum_{k \notin A} c(x_{k-1},x_k) = \sum_{k \notin A} d_{cc}(x_{k-1},x_k) \geq d_{cc}(y_{|A|},y_{|A|+1}) \geq c(y_{|A|},y_{|A|+1}). \label{e:Ac-c}
  \end{equation}
  Thus, by 
  \eqref{e:A-c} and \eqref{e:Ac-c}
  we get that
  \[
    c({\bf y}) = \sum_{k=1}^{|A|+1} c(y_{k-1},y_k) 
    \leq
    \sum_{k=1}^N c(x_{k-1},x_k) = c(\x).
  \]
\end{proof}

\begin{lemma}\label{lma:uniformlowerbounds-1}
 There exists a continuous function $f \colon [1,4] \to [\frac12,1]$ with the properties that
 $f(t)>\frac12$ for all $t \in (1,4)$ and 
 \begin{equation}
  d_1(0,(0,0,t4^{-n})) \ge f(t) d_{cc}(0,(0,0,t4^{-n})) \label{e:f-lower-bound}
 \end{equation}
 for all $n \in \N$ and $t \in (1,4)$.   
\end{lemma}

\begin{proof}
  We claim that
  \[
    f(t) = \min \left( \frac{1}{\sqrt{t}}
    , \frac{1}{2}\sqrt{ \frac{2t}{t+1}} \right)
  \]
  works.  It is immediate from definition that $f(t) > 1/2$ for $t \in (1,4)$.

  Let $\x = (x_0,x_1,\dots,x_N)$ be an itinerary from $0$ to $(0,0,t 4^{-n})$ where $t \in (1,4)$. 
  By Lemma~\ref{lma:vertical-itinerary}, we may suppose that $x_{i+1}^{-1}x_i \in {\rm Z}({\mathbb H})$. 
  Let $\ell_k$ be the absolute value of the $z$-coordinate of $x_k^{-1}x_{k-1}$.
Let   $\ell_M$ be the maximum of the $\ell_k$'s.
  Then we have that
  \begin{equation}
    \sum \ell_k \geq t 4^{-n}. \label{e:dk-lower}
  \end{equation}
    Suppose first that $\ell_M \geq 4^{-n+1}$, then
  \[
   c(\x) \geq \frac{1}{2} \ell_M^{1/2} \geq 2^{-n} = \frac{1}{\sqrt{t}} \sqrt{t} 2^{-n}\geq f(t) d_{cc}(0,(0,0,t4^{-n})),
  \]
  and we are done.
  Then suppose $\ell_M \in \left[ \frac{1+t}{2} 4^{-n}, 4^{-n+1} \right)$.  Then as $t \in (1,4)$, we get that $(x_{M-1},x_{M}) \notin \S$.  This gives
  \[
    c(\x) \geq \ell_M^{1/2} \geq \sqrt{ \frac{1+t}{2} } 2^{-n} \geq \sqrt{\frac{1 + t}{2t}} \sqrt{t} 2^{-n} \ge \frac{1}{\sqrt{t}}\sqrt{t}2^{-n},
  \]
  and we are done.
  Finally, suppose that  $\ell_M < \frac{1+t}{2} 4^{-n}$.
  By maximality of $\ell_M$ we then have $\ell_k < \frac{1+t}{2} 4^{-n}$ for all $k$.  Thus we have that
  \[
  2 c(\x) \geq  \sum_{k=1}^N \ell_k^{1/2} \geq \sqrt{\frac{2}{t+1}} 2^n \sum_{k=1}^N \ell_k \overset{\eqref{e:dk-lower}}{\geq} \sqrt{ \frac{2t}{t+1}} \sqrt{t}2^{-n},
  \]
  and we are done.
\end{proof}

\begin{lemma} \label{lma:uniformlowerbounds-2}
  For all $n \in a^{-1}(\{0\})$ and $t \in (\frac{1}{2},2)$, we have that
  \begin{equation}
    d_1(0,(0,0,t4^{-n})) \ge \frac{1}{\sqrt{3}} d_{cc}(0,(0,0,t4^{-n})). \label{e:sqrt3-lower}
  \end{equation}
\end{lemma}

\begin{proof}
  The proof is largely analogous to the proof of Lemma~\ref{lma:uniformlowerbounds-1}.

  Let $(x_0,x_1,\dots,x_N)$ be an itinerary from $0$ to $(0,0,t 4^{-n})$ where $t \in (1/2,2)$ and assume that $a(n) = 0$.  By Lemma~\ref{lma:vertical-itinerary}, we may suppose that $x_{i+1}^{-1}x_i \in {\rm Z}({\mathbb H})$.  
    Let $\ell_k$ and   $\ell_M$ be as in Lemma~\ref{lma:uniformlowerbounds-1}, so that we have \eqref{e:dk-lower}.

  Suppose first that $\ell \geq 4^{-n+1}$, then
  \[
c(\x) \geq \frac{1}{2} \ell^{1/2} \geq 2^{-n} = \frac{1}{\sqrt{2}} \sqrt{t} 2^{-n} 
 \ge \frac{1}{\sqrt{3}} d_{cc}(0,(0,0,t4^{-n})),
  \]
  and we are done.
  Then suppose $\ell \in \left[ \frac{1+4t}{8} 4^{-n}, 4^{-n+1} \right)$.  Then as $t \in (1/2,2)$ and $a(n) = 0$, we get that $(x_{M-1},x_M) \notin \S$.  This gives
  \[
  c(\x) \geq  \ell^{1/2} \geq \sqrt{ \frac{1+4t}{8} } 2^{-n} \geq \frac{1}{\sqrt{2}} \sqrt{t} 2^{-n},
  \]
  and we are done.
  Finally, suppose that  $\ell < \frac{1+4t}{8} 4^{-n}$.
  Thus, we have that $\ell_k < \frac{1+4t}{8} 4^{-n}$ for all $k$.  Therefore
  \begin{equation}
  2c(\x) \geq   \sum_{k=1}^N \ell_k^{1/2} \geq \sqrt{\frac{8}{4t+1}} 2^n \sum_{k=1}^N \ell_k \overset{\eqref{e:dk-lower}}{\geq} \sqrt{ \frac{8t^2}{4t+1}} 2^{-n}
    \ge \frac{2}{\sqrt{3}}\sqrt{t}2^{-n},
  \end{equation}
  and we are done.
\end{proof}

\begin{lemma}\label{lma:uniformupperbound}
  For every $\epsilon > 0$, there exists some $\eta \in (0,1/2)$ such that if $|t| < \eta$ and $a(n) = 1$, then
  \[
    d_1(0,(0,0,(1+t) 4^{-n})) \leq \left( \frac{1}{2} + \epsilon \right) d_{cc}(0,(0,0,(1+t) 4^{-n})).
  \]
\end{lemma}

\begin{proof}
  One has that
  \[
    d_{cc}(0,(0,0,(1+t) 4^{-n})) = \sqrt{1+t} d_{cc}(0,(0,0,4^{-n})).
  \]
  Consider the itinerary $\x = (0,4^{-n},(1+t)4^{-n})$.  Then
  \[
    c(\x) = \frac{1}{2} d_{cc}(0,(0,0,4^{-n})) + d_{cc}(0,(0,0,t 4^{-n})) = \left( \frac{1}{2} + \sqrt{|t|} \right) d_{cc}(0,(0,0, t4^{-n})).
  \]
  Thus, we need that
  \[
    \frac{1}{2} + \sqrt{|t|} \leq \left( \frac{1}{2} + \epsilon \right) \sqrt{1+t}.
  \]
  One sees easily that by taking $\eta$ small enough, we can satisfy this inequality.
\end{proof}

With the help of the above lemmas we conclude by proving:

\begin{proposition}
 No blow-up of $d_1$ is self-similar.
\end{proposition}
\begin{proof}
Assume to the contrary that there exists a sequence $(\lambda_j)_{j \in \N}$, with $\lambda_j \to 0$ such that the distances
$$(p,q)\mapsto \dfrac{1}{\lambda_j} d_1(\delta_{\lambda_j} (p),\delta_{\lambda_j} (q)) $$
converge pointwise to some $\rho$, and the distance $\rho$ is self-similar with some constant $\lambda > 1$. 

 Let us now find a contradiction by using the assumed self-similarity.
 For this purpose let us first take a point $(0,0,s^2) \in \mathbb H$ appearing as limit of points to which there is a shortcut from the origin.
 In other words, take
 \begin{equation}\label{14598015}
  s \in [1,2^4] \cap \bigcap_{j=1}^\infty\overline{\bigcup_{i \geq j} \{\lambda_i^{-1}2^{-4(k+1)}\,:\, k \in \N\}}.
 \end{equation}
 Then we indeed have
 \[
  \rho(0,(0,0,s^2)) = \lim_{j \to \infty}\frac{1}{\lambda_j}d_1(0,(0,0,\lambda_j^2s^2)) = \frac12d_{cc}(0,(0,0,s^2)),
 \]
 since $a(4(k+1))= 1$ for all $k \in \N$.

 Let us then use the function $f$ of Lemma~\ref{lma:uniformlowerbounds-1}
 to show that there exists $\ell \in \N$ such that $\lambda  = 2^\ell$.
 Supposing this is not the case, we have $\lambda = t2^\ell$ for some $t \in (1,2)$ and $\ell \in \N$.
 By \eqref{14598015} $s$ is of the form $s=\lim_{m\to \infty} \lambda_{i_m}^{-1} 2^{-4(k_m+1)}$, with $i_m, k_m\to \infty$.
 Then, by the continuity of the function $f$ we have
 \begin{eqnarray*}
  \rho(0,(0,0,\lambda^2s^2)) & = &\lim_{j  \to \infty}\frac{1}{\lambda_j}d_1(0,(0,0,\lambda^2\lambda_j^2s^2))\\
& = &\lim_{m  \to \infty}\frac{1}{\lambda_{i_m}} 
d_1(0,(0,0,  t^2 
\left(  \dfrac{\lambda_{i_m} s}{ 2^{-4(k_m+1)}}\right)^2  
 4^\ell  4^{-4(k_m+1)}\lambda^2\lambda_j^2s^2))
\\
  & \ge& 
  \lim_{m  \to \infty}\frac{1}{\lambda_{i_m}} 
 f\left(t^2   \left(  \dfrac{\lambda_{i_m} s}{ 2^{-4(k_m+1)}}\right)^2\right ) d_{cc}(0,(0,0,\lambda^2\lambda_j^2s^2))\\
&=&  f(t^2)\lim_{j  \to \infty}\frac{1}{\lambda_j}d_{cc}(0,(0,0,\lambda^2\lambda_j^2s^2))\\
  & = &f(t^2)\lambda d_{cc}(0,(0,0,s^2)) > \frac12\lambda d_{cc}(0,(0,0,s^2)) = \lambda \rho(0,(0,0,s^2))
  ,
 \end{eqnarray*}
 contradicting the fact that $\rho$ is self-similar with the dilation $\lambda$.

 Therefore $\lambda = 2^\ell$ for some $\ell \in \N$. Now we employ the properties of the function $a$.
 Let $m \in \N$ be the constant from Lemma~\ref{lma:norepeat}. Since $\rho$ is self-similar with factor $2^\ell$,
 it is self-similar also with factor $2^{\ell m}$.
 By Lemma~\ref{lma:uniformupperbound}, we have that there exists some $\eta$ such that $(1+\eta)^N = 4$ for some $N \in \N$ and if $a(n) = 1$, then
 \begin{equation}
   d_1(0,(0,0,(1+t)4^{-n})) \leq 0.51 d_{cc}(0,(0,0,(1+t)4^{-n})), \qquad \forall t \in (-\eta,\eta). \label{e:upperbound}
 \end{equation}
 
 Take $j_0 \in \N$ such that for all $j \ge j_0$ we have
 \begin{equation}
  \frac{\rho(0,(0,0,4^i(1+\eta)^ks^2))}{\lambda_j^{-1} d_1(0,(0,0,4^i(1+\eta)^k\lambda_j^2s^2))} \in \left(1 - \frac{1}{100}, 1 + \frac{1}{100} \right), \label{e:precision}
 \end{equation}
 for all $(i,k) \in \{0,1,\dots,2 m\ell\} \times \{0,\dots,N-1\}$.  Let $n \in \Z$ be such that $(\lambda_j s)^2 \in \left[ \frac{1}{2} 4^{-n}, 2 \cdot 4^{-n}\right)$.  By Lemma~\ref{lma:norepeat}, we have that there exists some $i \in \{0,\dots,m \ell\}$ such that $a(n+i) \neq a(n+ i + m\ell)$.  We may suppose without loss of generality that $a(n+i) = 1$ so $a(n+i+m\ell) = 0$.

 We have that there exists some $k \in \{0,\dots,N-1\}$ such that $4^i(1+\eta)^k \lambda_j^2 s^2 \in ((1-\eta) 4^{n+i},(1+\eta) 4^{n+i})$.  Thus, because $a(n+i) = 1$, we have that
 \begin{equation}\label{eq:finalcontra}
    \rho(0,(0,0,4^i (1+\eta)^k s^2)) \overset{\eqref{e:upperbound} \wedge \eqref{e:precision}}{\leq} 0.52 d_{cc}(0,(0,0,4^i (1+\eta)^k s^2)).
 \end{equation}
 On the other hand, because $a(n+i+m\ell) = 0$ and $$4^{i+m\ell}(1+\eta)^k \lambda_j^2 s^2 \in \left( \frac{1}{2} 4^{n+i+m\ell}, 2 \cdot 4^{n+i+m\ell} \right),$$ we have that
 \begin{eqnarray*}
   \rho(0,(0,0,4^{i+m\ell} (1+\eta)^k s^2)) &\overset{\eqref{e:sqrt3-lower} \wedge \eqref{e:precision}}{\geq} &\frac{99}{100} \frac{1}{\sqrt{3}} d_{cc}(0,(0,0,4^{i+m\ell} (1+\eta)^k s^2)) \\
   &\geq& 0.55 d_{cc}(0,(0,0,4^{i+m\ell}(1+\eta)^k s^2)), \\
   &=& 0.55 \cdot 2^{m\ell} d_{cc}(0,(0,0,4^i(1+\eta)^k s^2)).
 \end{eqnarray*}
 Then by the self-similarity of $\rho$ with ratio $2^{m\ell}$ we have
  \begin{eqnarray*}
   \rho(0,(0,0,4^i (1+\eta)^k s^2)) & = & 2^{-m\ell}\rho(0,(0,0,4^{i+m\ell} (1+\eta)^k s^2))\\
   & \ge & 0.55 d_{cc}(0,(0,0,4^i(1+\eta)^k s^2)).
  \end{eqnarray*}
  This contradicts \eqref{eq:finalcontra}.
 \end{proof}

 In order to obtain the distance $d_2$ of Theorem \ref{thm:fail}, we use only a subset of shortcuts used in the definition
 of the distance $d_1$.
 Let $D_n$ denote the centers of the dyadic cubes in $\R^2$ of sidelength $2^{-n}$.

Define the level $n$ shortcuts as the symmetrization of
\[
 \tilde\S_n = \{((x,y,z),(x,y,z)q) : (x,y) \in D_n, z \in \R, q = (0,0,\pm 4^{-n})\}.
\]
We then construct the set of shortcuts as
\[
  \tilde\S = \bigcup_{n \in a^{-1}(\{1\})} \tilde\S_n.
\]
As in the construction of $d_1$, the cost function $\tilde c : \mathbb{H} \times \mathbb{H} \to [0,\infty)$
for $(p,q) \in \tilde\S$ is 
\[
  \tilde c(p,q) = \frac{1}{2} d_{cc}(p,q).
\]
The distance $d_2$ is then obtained as the distance $d$ in \eqref{eq:ddef}, but now with using $\tilde c$.
Since $\tilde\S \subset \S$ and thus $\tilde c \ge c$, we have
\[
  \frac{1}{2} d_{cc} \leq d_1 \leq d_2 \leq d_{cc}. \label{e:metric-compare}
\]

We will also need the following lemmas.

\begin{lemma} \label{l:off-grid}
  There exists some absolute $\delta > 0$ so that if for any $n \in \N$, if $(x,y) \in B(D_n + (2^{-n-1},2^{-n-1}),\delta 2^{-n})$ and $t \in (1/2,2)$, then
  \[
    d_2((x,y,z),(x,y,z + t4^{-n})) \geq \frac{1}{\sqrt{3}} d_{cc}(0,(0,0,t4^{-n})), \qquad \forall z \in \R.
  \]
\end{lemma}

\begin{proof}
  Let $(a,b) \in D_n$ be so that $(x,y) \in B((a,b),\delta 2^{-n})$ ($\delta$ to be chosen later) and let $\x = (x_0,...,x_N)$ be an itinerary from $(x,y,z)$ to $(x,y,z + t4^{-n})$.  Note that
  \[
    d_{cc}(0,(0,0,t4^{-n})) = \frac{2}{\sqrt{\pi}} \sqrt{t 4^{-n}} \leq \sqrt{\frac{8}{\pi}} 2^{-n}.
  \]

  Suppose first that there is some $\pi(x_j) \notin B((a,b),2^{-n-1})$.  As any non-vertical movement is not a shortcut, we have from the fact that $\pi(x_0) = \pi(x_N) \in B((a,b),\delta 2^{-n})$ that if we choose $\delta$ sufficiently small, then
  \[
    c(\x) \geq (1-2\delta) 2^{-n} \geq \frac{1}{\sqrt{3}} \sqrt{\frac{8}{\pi}} 2^{-n} \geq \frac{1}{\sqrt{3}} d_{cc}(0,(0,0,t4^{-n})).
  \]
  This would prove the statement of the lemma.  Thus, we may suppose that the projection of $\x$ under $\pi$ does not go outside $B((a,b),2^{-n-1})$.

  But now the proof is reduced to that of the proof of Lemma \ref{lma:uniformlowerbounds-2}.  Indeed, by the hypothesis of this subcase, the itinerary $\x$ cannot contain any level $n$ shortcuts. 
\end{proof}

\begin{lemma} \label{l:grid-upper}
  For every $\epsilon > 0$ there exists $\eta \in (0,1/2)$ so that if $|t| < \eta$ and $a(n) = 1$, then for all $(x,y) \in B(D_n,\eta 2^{-n})$ and $z \in \R$ we have
  \[
    d_2((x,y,z),(x,y,z + (1+t)4^{-n})) \leq \left( \frac{1}{2} + \epsilon \right) d_{cc}(0,(0,0,(1+t)4^{-n})).
  \]
\end{lemma}

This follows by essentially the same proof as Lemma \ref{lma:uniformupperbound}.

Since the shortcuts are horizontally located on the centers of the dyadic cubes, we have from an easy argument using Lemmas \ref{l:off-grid} and \ref{l:grid-upper} that no blow-up of $d_2$ is left-invariant.

In order to see that no blow-up of $d_2$ is self-similar we argue similarly as for the distance $d_1$.
First suppose that a blow-up is self-similar with some constant $\lambda > 1$.
This time, instead of finding a limit point $(0,0,s^2)$ of shortcuts from the origin, we find by compactness a pair of points
$(x,y,z)$, $(x,y,z + s^2)$ in $B_{d_{cc}}(0,2^5)$ with $s \in [1,2^4]$
appearing as the limit of endpoints of a sequence of shortcuts. This is possible after passing to a subsequence
because $a(4(n+1)) = 1$ for all $n$.  Observe that Lemma \ref{lma:uniformlowerbounds-1}
holds also for $d_2$ since $d_2 \ge d_1$. As in the case of $d_1$, it then follows
via Lemma \ref{lma:uniformlowerbounds-1} that $\lambda = 2^\ell$ for some $\ell \in \N$.
A contradiction with self-similarity then follows again by the properties of the function $a$.
This concludes the proof of 
Theorem \ref{thm:fail}.

\bibliography{general_bibliography-TEMP} 
\bibliographystyle{amsalpha}
\end{document}